\documentclass[11pt]{amsart}
\usepackage{amssymb,amsmath,txfonts,mathrsfs,mathtools,titletoc, tikz, stmaryrd}
\usepackage{color}
\usepackage{graphicx}
\usepackage{float}
\usepackage{appendix}
\usepackage{hyperref}

\usepackage[alphabetic]{amsrefs}
\newtheorem{theorem}{Theorem}[section]
\newtheorem{prop}[theorem]{Proposition}
\newtheorem{lemma}[theorem]{Lemma}
\newtheorem{remark}[theorem]{Remark}

\newtheorem{notation}[theorem]{Notation}
\newtheorem{definition}[theorem]{Definition}

\newtheorem{conj}[theorem]{Conjecture}

\numberwithin{equation}{section}

\def\pf{{\it Proof:}~}
\def\R{{\bf\mathbb R} }
\def\S{{\bf\mathbb S} }
\renewcommand{\d}{\mathrm{d}}
\newcommand{\bracket}[1]{\ensuremath{\left<#1\right>}}

\begin{document}

\title[Volume growth and positive scalar curvature]{Volume growth and positive scalar curvature}
\author{Guodong Wei, Guoyi Xu,  Shuai Zhang}
\address{Guodong Wei\\School of Mathematics (Zhuhai)\\Sun Yat-sen University, Zhuhai, Guangdong\\ P. R. China, 519082}
\email{weigd3@mail.sysu.edu.cn}
\address{Guoyi Xu\\Department of Mathematical Sciences\\Tsinghua University, Beijing\\P. R. China, 100084}
\email{guoyixu@tsinghua.edu.cn}
\address{Shuai Zhang\\ Department of Mathematical Sciences\\Tsinghua University, Beijing\\P. R. China, 100084}
\email{zhangshu22@mails.tsinghua.edu.cn}
\date{\today}
\date{\today}

\begin{abstract}
For three dimensional complete,  non-compact Riemannian manifolds with non-negative Ricci curvature and uniformly positive scalar curvature,  we obtain the sharp linear volume growth ratio and the corresponding rigidity.  
\\[3mm]
\end{abstract}
\thanks{G. Wei is partially supported by NSFC 12101619 and 12141106. G. Xu was partially supported by NSFC 12141103.}

\maketitle

\titlecontents{section}[0em]{}{\hspace{.5em}}{}{\titlerule*[1pc]{.}\contentspage}
	\titlecontents{subsection}[1.5em]{}{\hspace{.5em}}{}{\titlerule*[1pc]{.}\contentspage}
	\tableofcontents

\section{Introduction}

From the celebrated Bishop-Gromov Volume Comparison Theorem,  it is well-known that for any complete non-compact Riemannian manifold $(M^n, g)$ with Ricci curvature $Rc\geq 0$, the $n$-dim \textbf{asymptotic volume growth ratio} of $M^n$, denoted as\begin{align}
\mathrm{V}_{M^n, n}\vcentcolon= \lim_{r\rightarrow\infty}\frac{V(B_p(r))}{\omega_n r^n}, \nonumber 
\end{align}
is well-defined and is independent of the choice of $p$; where $B_p(r)$ is the open geodesic ball centered at $p\in M^n$ with radius $r> 0$,  $V(B_p(r))$ is the volume of $B_p(r)$,  and $\omega_n$ is the volume of unit ball in $\mathbb{R}^n$. Furthermore $\mathrm{V}_{M^n, n}\in [0, 1]$, and $\mathrm{V}_{M^n, n}= 1$ if and only if the manifold $M^n$ is isometric to $\mathbb{R}^n$ by the rigidity part of the Bishop-Gromov's volume comparison Theorem.

Furthermore,  uniformly positive scalar curvature $R$ possibly controls the volume growth better.  Gromov \cite[(5')]{Gromov-large-mflds}  made the following conjecture:
\begin{conj}[Volume conjecture]
{There is a universal constant $C(n)> 0$,  such that if $(M^n,  g)$ has $Rc\geq 0$ and the scalar curvature $R\geq 2$, then 
\begin{align}
\sup_{p\in M^n\atop r> 0} \frac{V(B_p(r))}{r^{n- 2}}\leq C(n).  \nonumber 
\end{align}
}
\end{conj}

\begin{remark}\label{rem sectional curv is non-nega}
{Note Gromov \cite[2.B.]{Gromov-large-mflds} also sketched a proof of the following result: If $(M^n, g)$ has the sectional curvature $K\geq 0$ and the scalar curvature $R\geq 1$,  then there exists some $C(n)> 0$ such that $\displaystyle V(B_p(r))\leq C(n)\cdot r^{n- 2}$.  
}
\end{remark}

There is an asymptotic version of Gromov's volume conjecture as follows:
\begin{conj}[Volume growth conjecture]\label{conj volume growth of Gromov}
{There is a universal constant $C(n)> 0$,  such that if $(M^n,  g)$ has $Rc\geq 0$ and $R\geq 2$, then 
\begin{align}
\varlimsup_{r\rightarrow \infty} \frac{V(B_p(r))}{r^{n- 2}}\leq C(n), \quad \quad \quad \forall p\in M^n. \nonumber 
\end{align}
}
\end{conj}

Following Gromov's strategy and combining Cheeger-Colding's theory of Ricci limit spaces from \cite{CC-I}, Anderson-Cheeger's compactness theorem in \cite{AC}, and Gromov's spherical Lipschitz bound theorem in \cite{Gromov-MI}, Zhu \cite{Zhu} succeeded in proving: If $(M^n,  g)$ has $Rc\geq 0, R\geq 2$ and $\displaystyle \inf_{q\in M^n}V(B_q(1))\geq \varepsilon> 0$, then $\displaystyle \varlimsup_{r\rightarrow\infty} \frac{V(B_p(r))}{r^{n- 2}}< \infty$ for any $p\in M^n$.

Using the property of some suitable positive harmonic function on $(M^3, g)$,  Conjecture \ref{conj volume growth of Gromov} for $n= 3$ is firstly proved by Munteanu-Wang \cite{MW-II} as follows: 
\begin{theorem}\nonumber 
{There is a universal constant $C> 0$,  such that if $(M^3,  g)$ has $Rc\geq 0$ and $R\geq 2$, then $\displaystyle \varlimsup_{r\rightarrow\infty} \frac{V(B_p(r))}{r}\leq C$ for any $p\in M^3$.
}
\end{theorem}

Later, Chodosh-Li-Stryker \cite{CLS} gave another proof of the volume growth conjecture for $n= 3$ by using Cheeger-Colding's almost splitting theorem \cite{CC-Ann} and $\mu$-bubble's diameter estimate.

Motivated to get a generalization of the Cohn-Vossen's inequality in higher dimension,  Yau has posed the following question \cite[Problem 9]{Yau}: For any complete Riemannian manifold $(M^n, g)$ with $Rc\geq 0$, and for any $p\in M^n$, is it true that 
\begin{align}
\lim\limits_{r\rightarrow \infty} \frac{\int_{B_p(r)}R}{r^{n- 2}} < \infty ? \nonumber 
\end{align}
Note that a positive answer to Yau's question implies Conjecture \ref{conj volume growth of Gromov}. 

Shi and Yau \cite{SY} verified Yau's conjecture on K\"ahler manifolds with additional assumption.  If there is $p\in (M^n, g)$ such that the corresponding exponential map $\exp_p: T_pM^n\rightarrow M^n$ is a diffeomorphism, we call $(M^n,  g)$ is a Riemannian manifold with a pole.  For $3$-dim Riemannian manifold $M^3$ with a pole and $Rc\geq 0$,  the third named author \cite{Xu-IOCII} shows that
$$\displaystyle \lim_{r\rightarrow\infty} \frac{\int_{B_p(r)}R}{r}= 8\pi (1- \mathrm{V}_{M^3, 3}).$$
This formula shows that general  $3$-dim complete,  noncompact manifolds with $Rc\geq 0$ possibly has some sharp asymptotic relationship between scalar curvature and volume growth.  In this paper,  we address this sharp asymptotic relationship reflected in Conjecture \ref{conj volume growth of Gromov}. 

Similar to the concept of asymptotic volume ratio,  for complete  non-compact Riemannian manifolds $(M^n, g)$ with $n\geq 3$, we define the \textbf{$(n-2)$-dim asymptotic volume growth ratio of $M^n$} as follows:
\begin{align}
\mathrm{V}_{M^n, n- 2}\vcentcolon= \varlimsup\limits_{r\rightarrow\infty}\frac{V(B_p(r))}{\omega_{n- 2}r^{n- 2}},  \quad \quad \quad \forall p\in M^n.  \nonumber 
\end{align}
It is easy to see that $\mathrm{V}_{M^n, n- 2}$ is well-defined and is independent of the choice of $p$.

In this paper,  we prove the sharp version of Conjecture \ref{conj volume growth of Gromov} for $n= 3$.  
\begin{theorem}\label{thm main-1 }
{For any complete non-compact Riemannian manifold $(M^3, g)$ with $Rc\geq 0$ and $R\geq 2$, 
\begin{enumerate}
\item[(1)].  if $(M^3, g)$ has only one end,  then $\displaystyle 0< \mathrm{V}_{M^3, 1}\leq 2\pi$,
\item[(2)]. if $(M^3,g)$ has at least two ends, then $(M^3,g)$ is isometric to a cylinder $S\times\mathbb{R}$ with $S$ being a closed surface of sectional curvature at least one, and $\displaystyle 0< \mathrm{V}_{M^3, 1}\leq 4\pi$. 
\item[(3)].  moreover $\mathrm{V}_{M^3, 1}= 4\pi$ if and only if $(M^3,g)$ is isometric to the cylinder $\mathbb{S}^2\times\mathbb{R}$,  where $\mathbb{S}^2$ is the unit round sphere in $\mathbb{R}^3$.
\end{enumerate}
}
\end{theorem}

\begin{remark}\label{rem uniqueness of rigidity}
{Meanwhile,  for any $0<b<1$,  we construct distinct noncompact manifolds $(M_i^3,g_i)$ with  $Rc\geq 0$ and $R\geq 2$,  where $i=1, 2$,  such that   $\displaystyle \mathrm{V}_{M_i^3, 1}=4\pi b$. See Remark \ref{rem two end models} and Remark \ref{rem one end models} for details.
}
\end{remark}

The proof of Theorem \ref{thm main-1 } follows the strategy of \cite{CLS} closely.  

From Cheeger-Colding's almost splitting theorem,  the $L$-neighborhood (where $L> 0$)  of some part of a ray (which is $[(k- 5)\cdot L,  (k+ 5)\cdot L]$ away from $p$,  where $k>> 1$),  is $\delta$-Gromov-Hausdorff close to,  the product of an interval with a length space (a cylindrical region); where $\delta> 0$ goes to $0$ as $k\rightarrow \infty$.  On the other hand,  from theory of $\mu$-bubble surfaces,  we can get the area upper bound of the $\mu$-bubble surfaces in the above almost cylindrical region in term of $L$.  

\begin{remark}\label{rem key observations}
{There are two key observations in this paper: 
\begin{enumerate}
\item[(A)].  The area of $\mu$-bubble surfaces is bounded by the area of unit round sphere $\mathbb{S}^2$,  as $L\rightarrow \infty$.  This bound is sharp in asymptotic sense,  where asymptotic means $L\rightarrow \infty$.
\item[(B)].  The area of the intersection of geodesic spheres with the neighborhood of the ray,  has an upper bound in term of the area bound of corresponding $\mu$-bubble surfaces.
\end{enumerate}
}
\end{remark}

When computing $\mathrm{V}_{M^3, 1}$,  we take $r= k\cdot L$ for fixed $L$,  and express the upper bound of $\mathrm{V}_{M^3, 1}$ as a limit  in terms of $\delta,  L$ with respect to $k\rightarrow \infty$(see (\ref{final ineq with L and delta})).  Finally we let $L\rightarrow\infty,  \delta\rightarrow 0$ to obtain the sharp linear volume growth ratio.  
 
In the end of this section,  we sketch the proof of Theorem \ref{thm main-1 } in more details as follows.

We firstly introduce the following notation, which is used in the rest of the paper. 
\begin{notation}\label{notation Ek and Nk}
{For a complete Riemannian manifold $(M^n, g)$ and $L> 0$,  if $(M^n, g)$ has only one end,  the unique unbounded connected component of $M^n- \overline{B_p(kL)}$ is denote by $E_k(p, L)$ . And when the context is clear (or $p, L$ are fixed), we use $E_k$ to denote $E_k(p, L)$ for simplicity. Furthermore, we define $N_k\vcentcolon= E_k- \overline{E_{k+ 1}}$. 
}
\end{notation}

\begin{enumerate}
\item[(1)]. Cutting the whole manifold into a sequence of cylindrical regions $\{N_k\}_{k\in \mathbb{Z}^+}$ with respect to the scale $L> 0$, reduce the asymptotic upper bound of geodesic ball's volume to the upper bound of the volume of $N_k$ (see Figure \ref{fig:cylindrical region in introduction}).
\item[(2)].  To bound the volume of $N_k$,  using the co-area formula,  we are reduced to get the upper bound of the `height' of $N_k$ (which is the oscillation of the distance between points in $N_k$ and $p$),  and the upper bound of the area of the intersection of geodesic sphere with $N_k$ (this is done in Section \ref{sec main proof}).  

\begin{figure}[H]
    \centering
    \includegraphics[width=0.5\linewidth]{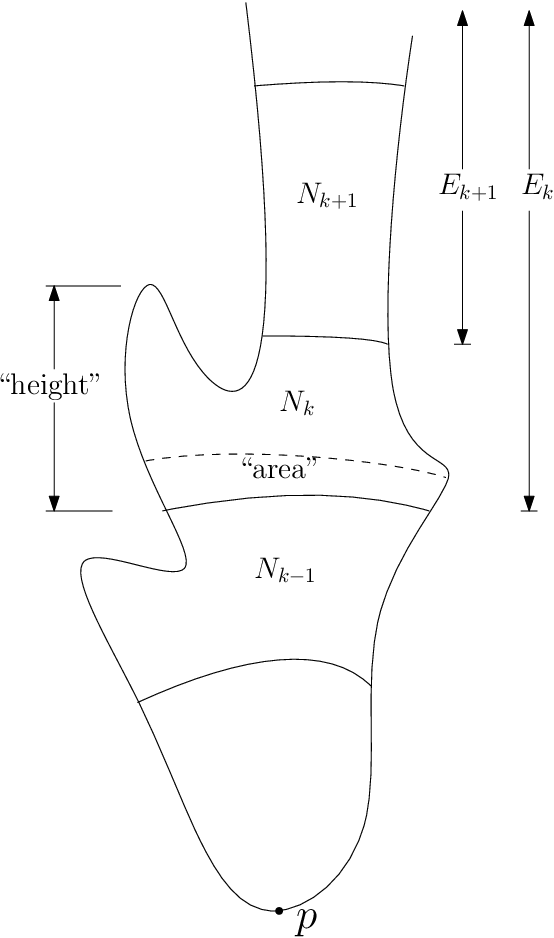}
    \caption{Picture for (1) and (2)}
    \label{fig:cylindrical region in introduction}
\end{figure}

\item[(3)].  The asymptotic equivalence between the `height' of $N_k$ and the scale $L$,  is showed in Section \ref{sec radius est for annulus}.  And this asymptotic equivalence relies on the uniform diameter bound of suitable subset of geodesic sphere,  which is the intersection of geodesic sphere with $N_k$.  Note the uniform diameter bound (denoted as $c$ in most of the rest argument) is independent of the large scale $L$ (see Figure \ref{fig:radius estimate in introduction}).
\begin{enumerate}
\item[(3.a)].  To prove the uniform diameter bound of the intersection of geodesic sphere with $N_k$,  we use Cheeger-Colding's almost splitting theorem to reduce it to the uniform diameter bound of corresponding,  separating $\mu$-bubble surfaces.  
\item[(3. b)].  The uniformly positive scalar curvature assumption and the second variation formula for $\mu$-bubble surface,  together a conformal version of Bonnet-Myers' type estimate (obtained in Section \ref{sec bubble hypersurfaces and diameter}),  are used to get such diameter bound of the above $\mu$-bubble surfaces in Section \ref{sec separating bubble surfaces}.
\end{enumerate}

\begin{figure}[H]
    \centering
    \includegraphics[width=\linewidth]{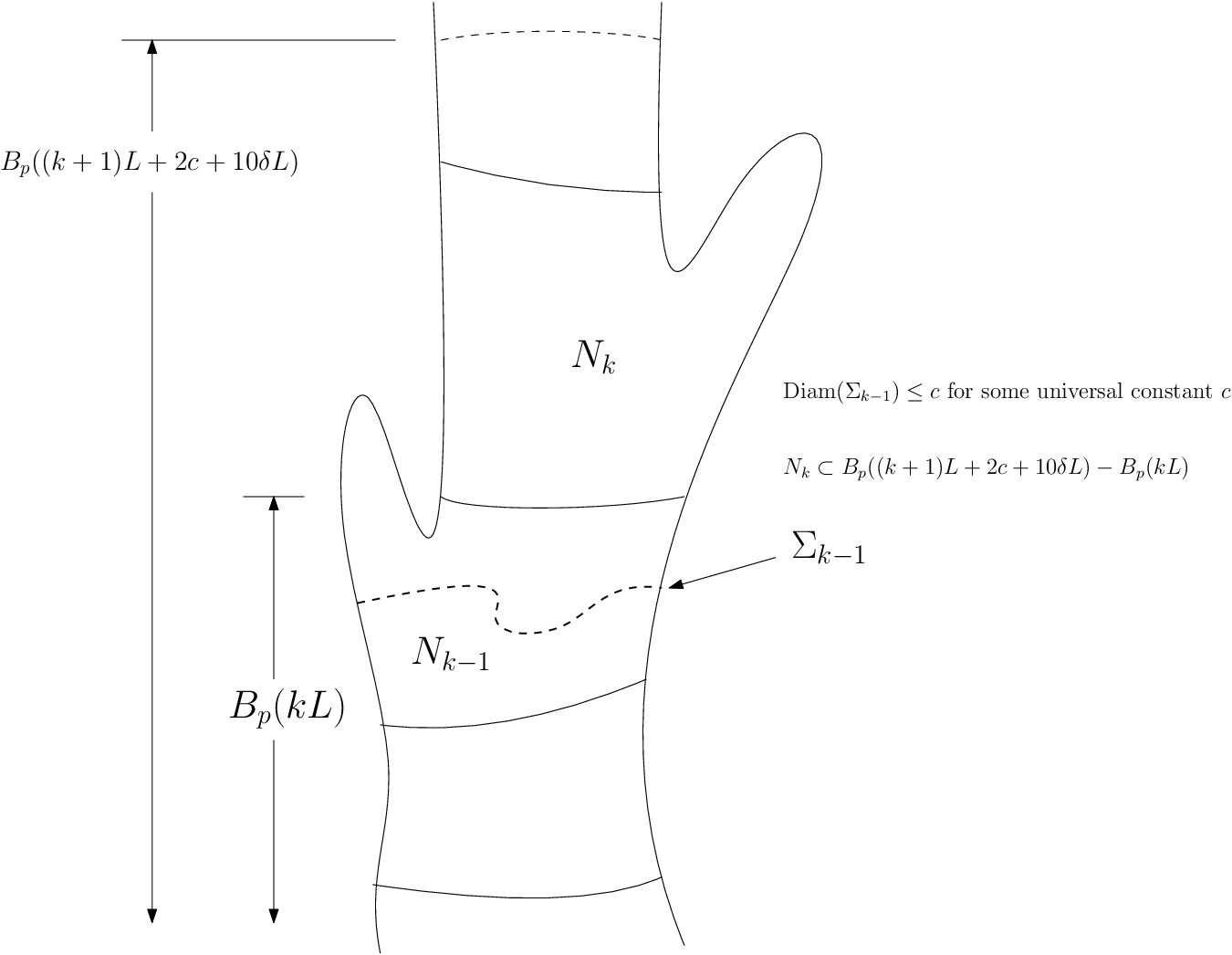}
    \caption{Picture for (3)}
    \label{fig:radius estimate in introduction}
\end{figure}

\item[(4)].  The upper bound of the area of the intersection of geodesic sphere with $N_k$ is obtained from the following two results:
\begin{enumerate}
\item[(4.a)] The volume comparison for suitable part of geodesic spheres and corresponding separating $\mu$-bubble surfaces,  is provided in Section \ref{sec area compa}.  This volume comparison result is in the similar spirit of Bishop-Gromov's volume comparison Theorem (also see \cite{HK}),  and holds for general dimensions (see Figure \ref{fig:area estimate in introduction}). 

\begin{figure}[H]
    \centering
    \includegraphics[width=0.7\linewidth]{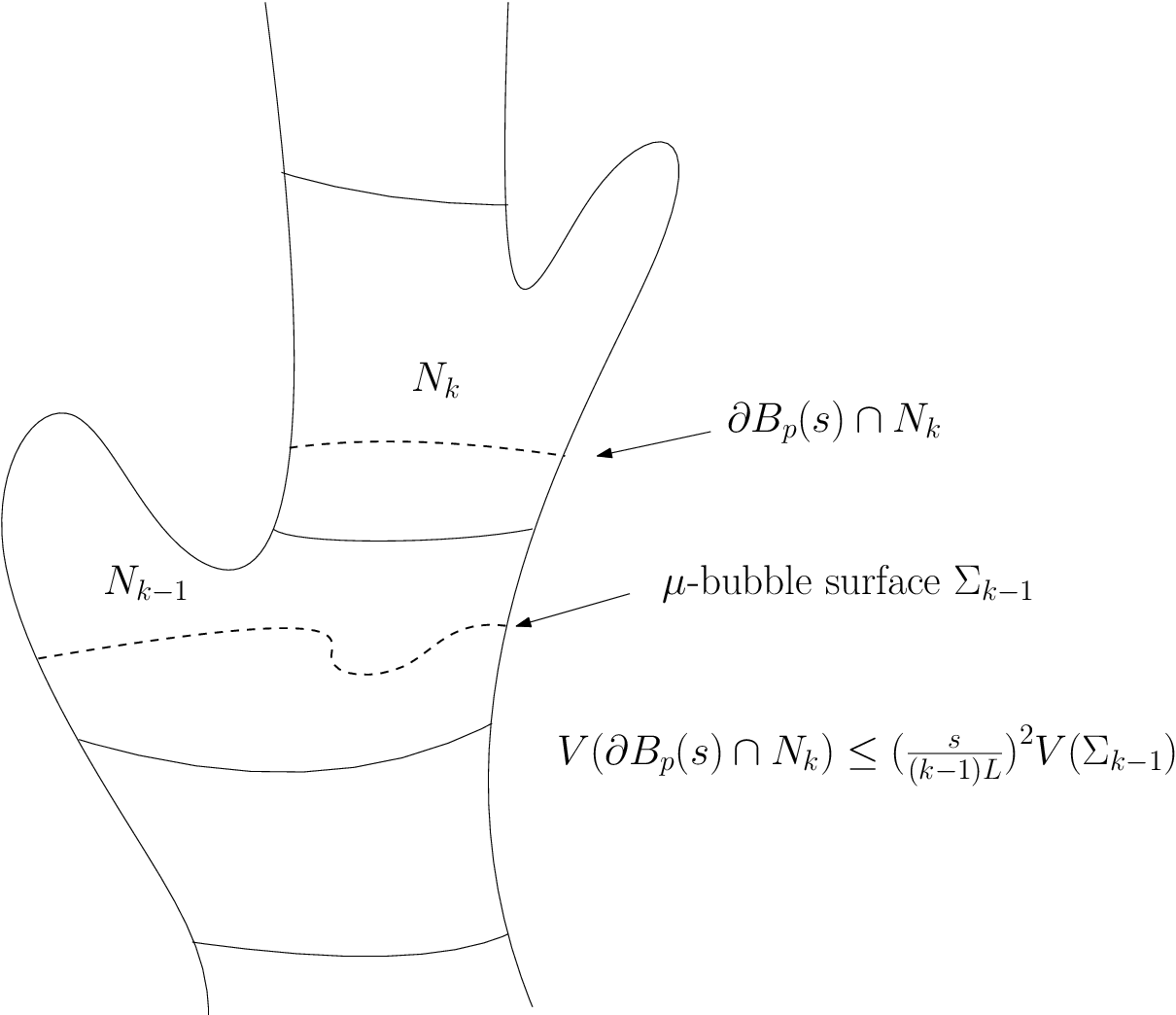}
    \caption{Picture for (4.a)}
    \label{fig:area estimate in introduction}
\end{figure}

\item[(4.b)] Based on some results about separating $\mu$-bubble surfaces in \cite{CLS} and \cite{CLS-0},  the area bound of separating $\mu$-bubble surfaces is obtained in Section \ref{sec separating bubble surfaces}.
\end{enumerate}
\end{enumerate}

\begin{remark}\label{rem key points}
{Another key point is: all estimates (including some `height' estimates of cylindrical region) related to area or volume obtained in this paper are asymptotically sharp with respect to $L\rightarrow \infty$.
}
\end{remark}

There are some smooth functions required in the construction of $\mu$-bubbles and corresponding $\mu$-bubble hypersurfaces, and we put the existence result of those functions into Appendix \ref{sec appe exist}. 

For readers' convenience, we provide a diagram as follows, which is the structure of the proof of Theorem  \ref{thm main-1 } (see Figure \ref{fig:structure of proof}).

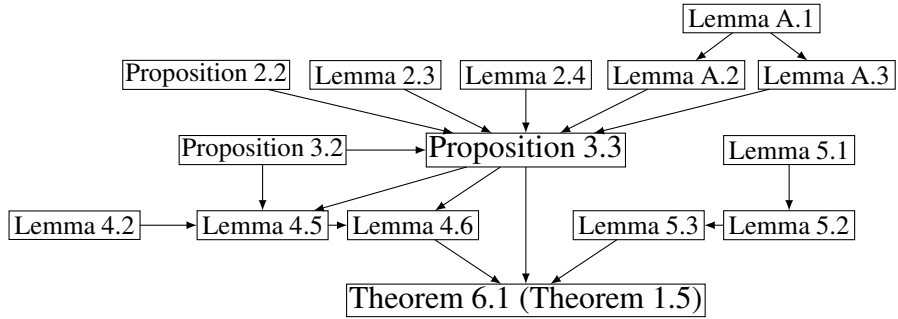
\begin{figure}[H]
    \centering
\begin{tikzpicture}
\node [draw,outer sep=0,inner sep=1,minimum size=10, font=\fontsize{12}{0}\selectfont] (v2) at (-4.5,1) {Theorem \ref{thm sharp volume growth}  (Theorem \ref{thm main-1 })};
\node [draw,outer sep=0,inner sep=1,minimum size=10, font=\fontsize{12}{0}\selectfont] (v1) at (-4.5,3) {Proposition \ref{prop volume of separating bubble surfaces}};

\node [draw,outer sep=0,inner sep=1,minimum size=10, font=\fontsize{10}{0}\selectfont] (v6) at (-8,3) {Proposition \ref{prop connected property of annulus}};
\node [draw,outer sep=0,inner sep=2,minimum size=10, font=\fontsize{10}{0}\selectfont] (v12) at (-10.5,2) {Lemma \ref{lem connected of some set}};
\node [draw,outer sep=0,inner sep=2,minimum size=10, font=\fontsize{10}{0}\selectfont] (v10) at (-8,2) {Lemma \ref{lem diam of partial geod spheres}};
\node [draw,outer sep=0,inner sep=2,minimum size=10, font=\fontsize{10}{0}\selectfont] (v5) at (-6,2) {Lemma \ref{lem annulus is not too far from p}};

\node [draw,outer sep=0,inner sep=2,minimum size=10, font=\fontsize{10}{0}\selectfont] (v11) at (-1,3) {Lemma \ref{lem det of positive semi definite matrix}};
\node [draw,outer sep=0,inner sep=2,minimum size=10, font=\fontsize{10}{0}\selectfont] (v3) at (-1,2) {Lemma \ref{lem volume of surface in term of polar Jac}};
\node [draw,outer sep=0,inner sep=2,minimum size=10, font=\fontsize{10}{0}\selectfont] (v4) at (-3,2) {Lemma \ref{lem surface area comparison}};

\node [draw,outer sep=0,inner sep=1,minimum size=10, font=\fontsize{10}{0}\selectfont] (v8) at (-8.75,4) {Proposition \ref{prop existence of bubbles}};
\node [draw,outer sep=0,inner sep=2,minimum size=10, font=\fontsize{10}{0}\selectfont] (v7) at (-6.5,4) {Lemma \ref{lem 1st and 2nd vari of bubble}};
\node [draw,outer sep=0,inner sep=2,minimum size=10, font=\fontsize{10}{0}\selectfont] (v9) at (-4.5,4) {Lemma \ref{lem the diameter upper bound of 2-dim surface-SY}};

\node [draw,outer sep=0,inner sep=2,minimum size=10, font=\fontsize{10}{0}\selectfont] (v14) at (-1.5,4.75) {Lemma \ref{lem existence of bump function}};
\node [draw,outer sep=0,inner sep=2,minimum size=10, font=\fontsize{10}{0}\selectfont] (v15) at (-2.5,4) {Lemma \ref{lem function for Caccioppoli set}};
\node [draw,outer sep=0,inner sep=2,minimum size=10, font=\fontsize{10}{0}\selectfont] (v16) at (-0.5,4) {Lemma \ref{lem existence of smooth function}};

\draw [-latex] (v1) edge (v2);
\draw [-latex] (v3) edge (v4);
\draw [-latex] (v4) edge (v2);
\draw [-latex] (v5) edge (v2);
\draw [-latex] (v1) edge (v5);
\draw [-latex] (v7) edge (v1);
\draw [-latex] (v8) edge (v1);
\draw [-latex] (v9) edge (v1);
\draw [-latex] (v10) edge (v5);
\draw [-latex] (v11) edge (v3);
\draw [-latex] (v6) edge (v10);
\draw [-latex] (v1) edge (v10);
\draw [-latex] (v6) edge (v1);
\draw [-latex ] (v12) edge (v10);
\draw [-latex] (v14) edge (v15);
\draw [-latex] (v14) edge (v16);
\draw [-latex] (v15) edge (v1);
\draw [-latex] (v16) edge (v1);
\end{tikzpicture}
  \caption{Structure of the proof of Theorem \ref{thm main-1 }}
   \label{fig:structure of proof}
\end{figure}

\section{The $\mu$-bubble surfaces and diameter estimate}\label{sec bubble hypersurfaces and diameter}

The soap bubble was firstly proposed by Schoen to Gromov \cite[Section $5\frac{5}{6}$]{Gromov-Macro}.  We recall its definition as follows.

\begin{definition}\label{def bubble and bubble surface}
{Assume $N$ is a smooth Riemannian manifold with boundary $\partial N= \partial_{-}N\cup \partial_{+}N$, where $\partial_{-}N\cap \partial_{+}N= \emptyset$ and $\partial_{\pm}N\neq \emptyset$. Fix a smooth function $u> 0$ on $N$ and a smooth function $h$ on $\mathring{N}$ with $h\rightarrow \pm \infty$ on $\partial_{\pm}N$, where $\mathring{N}$ is the interior of $N$. Choose a Caccioppoli set $\mathcal{C}\subseteq N$ such that
\begin{align}
\partial \mathcal{C}- \partial_+ N \ \text{is smooth},\quad \quad \partial \mathcal{C}- \partial_+ N\subseteq \mathring{N}\quad \quad \text{and} \quad \quad \partial_+ N\subseteq \partial \mathcal{C}. \label{assumption on Cacci set C}
\end{align}
Define $\mathscr{C}(N, \mathcal{C})$ as the set of all Caccioppoli sets $\Omega\subseteq N$ with $\Omega\triangle \mathcal{C}\subseteq \mathring{N}$.

We define $\mathcal{A}: \mathscr{C}(N, \mathcal{C})\rightarrow \mathbb{R}$ as follows:
\begin{align}
\mathcal{A}(W)= \int_{ \partial^* W- \partial N}u d\mathcal{H}^{n-1}- \int_N (\chi_W- \chi_{\mathcal{C}})\cdot h\cdot ud\mathcal{H}^{n},\quad \quad \quad \forall W\in \mathscr{C}(N, \mathcal{C}),\nonumber 
\end{align}
where $\partial^* W$ is the reduced boundary of $W$. 
 
 If $\displaystyle \mathcal{A}(\Omega)= \inf_{W\in \mathscr{C}(N, \mathcal{C})}\mathcal{A}(W)$ for some $\Omega\in \mathscr{C}(N, \mathcal{C})$, we call $\Omega$ as a \textbf{$\mu$-bubble with respect to $(h, u, \mathcal{C})$ in $(N, \partial_{-}N, \partial_{+}N)$} (See Figure \ref{figure: bubble-surface}). 

For any $\mu$-bubble $\Omega$ with respect to $(h, u, \mathcal{C})$ in $(N, \partial_{-}N, \partial_{+}N)$, we call any connected component of $\partial\Omega- \partial N$ as a \textbf{$\mu$-bubble hypersurface with respect to $(h, u, \mathcal{C})$ in $(N, \partial_{-}N, \partial_{+}N)$}. When $\partial\Omega- \partial N$ is two dimensional, we also call its connected components as $\mu$-bubble surfaces.
}
\end{definition}

\begin{figure}[H]
    \centering
    \includegraphics[width=0.8\linewidth]{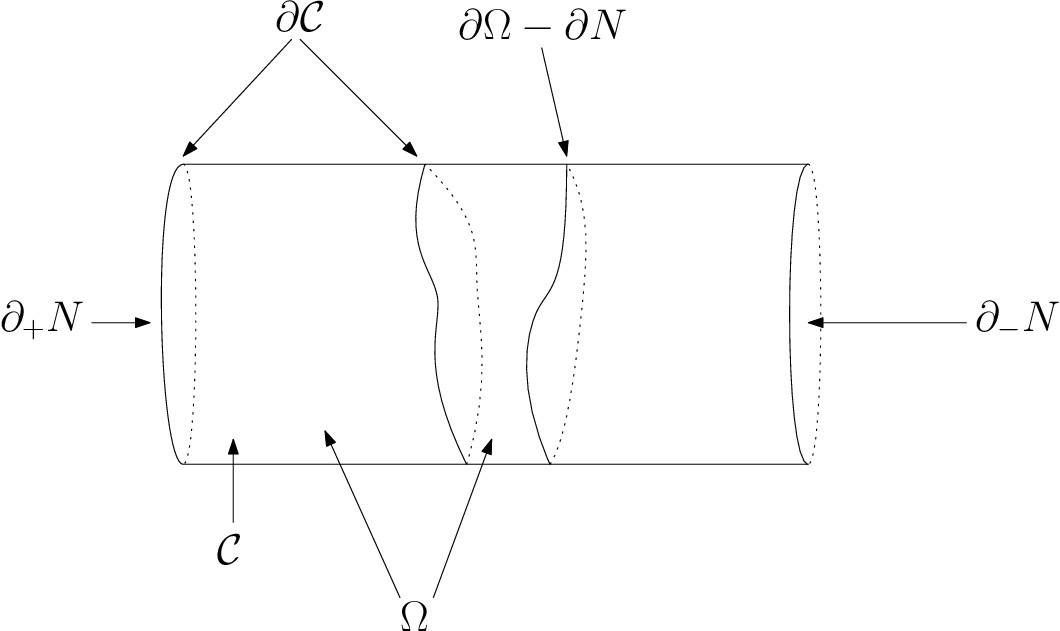}
    \caption{The $\mu$-bubble and $\mu$-bubble surface}
    \label{figure: bubble-surface}
\end{figure}

The following existence result of $\mu$-bubble is \cite[Proposition $12$]{CL} (also see \cite[Proposition $2.1$]{Zhu-bubble})
\begin{prop}\label{prop existence of bubbles}
{There is a $\mu$-bubble $\Omega$ with respect to $(h, u, \mathcal{C})$ in $(N, \partial_{-}N, \partial_{+}N)$, and any $\mu$-bubble hypersurface corresponding to $\Omega$ is a smooth hypersurface.
}
\end{prop}\qed

We use the second variation formula for $\mu$-bubble hypersurface $\Sigma$ repeatedly (see \cite[Lemma $13$ and $14$]{CL}), the statement is as follows.
\begin{lemma}\label{lem 1st and 2nd vari of bubble}
{If $\Sigma$ is a $\mu$-bubble hypersurface with respect to $(h, u, \mathcal{C})$ in $(N, \partial_{-}N, \partial_{+}N)$,  where $N$ is $n$-dim Riemannian manifold with $n\geq 3$; then for any $\psi\in C^\infty(\Sigma)$,  we have
\begin{align}
&\int_\Sigma |\nabla_\Sigma \psi|^2 u- \frac{1}{2}(R_N- R_\Sigma+ |\mathring{A}_\Sigma|^2)\cdot \psi^2 u+ (\Delta_N u- \Delta_\Sigma u)\psi^2 \nonumber \\
&-\frac{1}{2}u^{-1}\langle \nabla_N u, \vec{n}\rangle^2\psi^2- \frac{1}{2}\{h^2+ 2\langle \nabla_N h, \vec{n}\rangle\}\psi^2 u \geq 0;\label{2nd vari ineq-higher dim}
\end{align}
where $\mathring{A}_\Sigma$ is the traceless part of second fundamental form of $\Sigma\subseteq N$,  $\vec{n}$ is the unit normal vector of $\Sigma\subseteq N$,  and $R_N, R_\Sigma$ are the intrinsic scalar curvature of $N$ and $\Sigma$ respectively.
}
\end{lemma}\qed

For later height estimates of cylindrical region, we need the diameter bound of $\mu$-bubble surfaces.  Similar diameter bound is given in \cite[Lemma $16$]{CL}.  Our method follows closely \cite{SY-CMP} (also see \cite{SY-PSC}). 

\begin{lemma}\label{lem the diameter upper bound of 2-dim surface-SY}
{For compact $(\Sigma^2, g)$, if there is a smooth function $u> 0$ such that 
\begin{align}
-\Delta u\geq (K- K_\Sigma)u, \nonumber 
\end{align}
for some constant $K> 0$, then $\mathrm{Diam}(\Sigma)\leq \frac{2\pi}{\sqrt{3K}}$. 
}
\end{lemma}

\begin{remark}\label{rem PDE Bonnet-Myers thm}
{The sharp upper bound of $\mathrm{Diam}(\Sigma)$ in the above lemma is unknown.  Comparing the well-known Bonnet-Myers Theorem,  we can view Lemma \ref{lem the diameter upper bound of 2-dim surface-SY} as a type of conformal Bonnet-Myers Theorem for surfaces.   
}
\end{remark}

\pf
{\textbf{Step (1)}. Assume $\mathrm{Diam}(\Sigma)= d_g(p, q)$, where $p, q\in \Sigma$. Define $I[\gamma]$ by 
\begin{align}
I[\gamma]= \int_\gamma u ds, \nonumber 
\end{align}
where $ds$ is arclength along $\gamma$ and $\gamma$ is any curves from $p$ to $q$. Let $\gamma_0$ be one minimizer of $I[\cdot]$. 

As in the proof of \cite[Proposition $1$]{SY-CMP} (also see \cite[Proposition $2.2$]{LZ}), let $V(s)= f(s)\cdot v(s)$ be the variation vector field along $\gamma(s)$, where $v(s)$ is the unit normal vector field along $\gamma$. The first variation formula yields
\begin{align}
 u(\gamma_0(s))\nabla_{\partial s}\partial s- (\nabla u)^\perp= 0, \quad \text{along}\ \ \gamma_0.\label{mean curv vanish}
\end{align}

The non-negativity of the  second variation of $I$ at $\gamma_0$ gives
\begin{align}
0\leq \delta^2I[\gamma_0]= \int_{\gamma_0} [D^2u(v, v)- 2u\cdot H^2]\cdot f^2+ u\cdot (|f'|^2- K_\Sigma\cdot f^2)ds, \label{2nd vari non-negat}
\end{align}
where $H$ is the mean curvature of $\gamma\subseteq \Sigma$.  

Note $\displaystyle \Delta_\Sigma u= D^2u(v, v)+ D^2u(\partial s, \partial s)$, combining (\ref{mean curv vanish}), we get
\begin{align}
\Delta_\Sigma u&= D^2u(v, v)+ D^2u(\partial s, \partial s)= D^2u(v, v)+ u''- \nabla_{\nabla_{\partial s}\partial s}u \nonumber \\
&= D^2u(v, v)+ u''- H^2\cdot u, \label{Hessian of u}
\end{align}
where  $u'\vcentcolon= \frac{d}{d s}u(\gamma_0(s))$.

Plugging (\ref{Hessian of u}) into (\ref{2nd vari non-negat}), using the equation of $u$ in the assumption, we obtain
\begin{align}
0&\leq \int_{\gamma_0} [\Delta_\Sigma u- u''(s)- u\cdot H^2]\cdot f^2+ u\cdot (|f'|^2- K_\Sigma\cdot f^2)ds\nonumber \\
&\leq \int_{\gamma_0} - u''(s)\cdot f^2+ u\cdot (|f'|^2- K f^2)ds. \nonumber 
\end{align}

Assume $l= \mathrm{length}(\gamma_0)$,  then integration by parts yields
\begin{align}
\int_0^l u(f')^2+ 2fu'f'- Kuf^2 ds\geq 0, \quad \quad \forall f\in C_0^\infty[0, l]. \label{integ ineq-main}
\end{align}

\textbf{Step (2)}. Put $f(s)= u^{-\frac{1}{2}}(\gamma_0(s))\sin(\frac{\pi}{l}s)$ into (\ref{integ ineq-main}), we get
\begin{align}
&\quad \int_0^l -\frac{3}{4}[(\ln u)']^2 \sin^2(\frac{\pi}{l}s)+ \frac{\pi}{l}(\ln u)'\sin(\frac{\pi}{l}s)\cdot \cos(\frac{\pi}{l}s) ds \nonumber \\
&\geq \int_0^l \left(K\sin^2(\frac{\pi}{l}s)- (\frac{\pi}{l})^2 \cos^2(\frac{\pi}{l}s)\right) ds. \label{integ ineq-intermediate}
\end{align}

By Cauchy-Schwarz inequality, we have
\begin{align}
-\frac{3}{4}[(\ln u)']^2 \sin^2(\frac{\pi}{l}s)+ \frac{\pi}{l}(\ln u)'\sin(\frac{\pi}{l}s)\cdot \cos(\frac{\pi}{l}s) \leq \frac{1}{3}\cdot (\frac{\pi}{l})^2\cos^2(\frac{\pi}{l}s). \label{CS-ineq}
\end{align}

Plugging (\ref{CS-ineq}) into (\ref{integ ineq-intermediate}), we obtain
\begin{align}
\frac{4}{3}\cdot (\frac{\pi}{l})^2\int_0^l\cos^2(\frac{\pi}{l}s) ds\geq K\int_0^l \sin^2(\frac{\pi}{l}s) ds. \nonumber 
\end{align}
This implies $l\leq \frac{2\pi}{\sqrt{3K}}$. 
}
\qed

\section{The separating $\mu$-bubble surfaces}\label{sec separating bubble surfaces}

\begin{definition}\label{def seperating property}
{If a smooth compact hypersurface $\Sigma_k\subseteq E_k- \overline{E_{k+ 1}}$ satisfies the following property: ``for any continuous curve $\gamma: [0, 1]\rightarrow \overline{E_k}- E_{k+ 1}$ with $\gamma(0)\in \partial E_k, \gamma(1)\in \partial E_{k+ 1}$, we have $\gamma\cap \Sigma_k\neq \emptyset$", then we say that \textbf{$\Sigma_k$ separates $\partial E_k$ from $\partial E_{k+ 1}$}. 
}
\end{definition}

In the applications of $\mu$-bubble surface, we require that it satisfies the above separating property. The following proposition is important to obtain such separating property.

\begin{prop}\label{prop connected property of annulus}
{If $(M^n,  g)$ is a complete Riemannian manifold with only one end and is simply connected,  then there is $k_0\in \mathbb{Z}^+$, such that $\overline{E_k}- E_{k+ 1}$ and $\partial E_k$ are both connected for $k\geq k_0$. 
}
\end{prop}

\pf
{See \cite[Proposition $3.2$]{CLS-0}. 
}
\qed

Now we present the main result of this section, which provides us the separating $\mu$-bubble surfaces with upper bound of area and diameter in cylindrical region $E_k- \overline{E_{k+ 1}}$ when $k$ is big enough. One crucial thing is that the upper bound of the area is asymptotically sharp with respect to the large scale $L\rightarrow \infty$.
\begin{prop}\label{prop volume of separating bubble surfaces}
{Assume $(M^3, g)$ is a complete non-compact Riemannian manifold with $R\geq 2$ and only one end, also it is simply connected and $p\in M^3$.  For any $\varepsilon>0$ and $L> \sqrt{2}\pi(1+\varepsilon)$, there is $k_0= k_0(\varepsilon, L)\in \mathbb{Z}^+$; such that for any $k\geq k_0$, one can find a smooth compact surface $\Sigma_k\subseteq E_k- \overline{E_{k+ 1}}$, which satisfies
\begin{enumerate}
\item[(1)]. $\Sigma_k$ separates $\partial E_k$ from $\partial E_{k+ 1}$;
\item[(2)]. $\displaystyle V(\Sigma_k)\leq \frac{4\pi}{1- \frac{2\pi^2(1+\varepsilon)^2}{L^2}}$. 
\item[(3)]. $\mathrm{Diam}(\Sigma_k)\leq \frac{2\pi}{\sqrt{3(1- \frac{2\pi^2(1+ \varepsilon)^2}{L^2})}}$.
\end{enumerate}
}
\end{prop}

\begin{remark}\label{rem separate}
{The separating property of suitable $\mu$-bubble surface $\Sigma_k$ (the property (1) above) was firstly observed in the argument of \cite[Lemma $5.4$]{CLS-0}, which is used for our later argument of Lemma \ref{lem surface area comparison}. 
}
\end{remark}

\pf
{\textbf{Step (1)}. It is obvious that $E_{k+ 1}\subseteq E_k$. Recall $N_k= E_k- \overline{E_{k+ 1}}$, and $\partial_+N_k= \partial E_k, \partial_-N_k= \partial E_{k+ 1}$. Then $\partial N_k= \partial_- N_k\sqcup \partial_+ N_k$. From Lemma \ref{lem function for Caccioppoli set}, we can choose a smooth function $f\in C^\infty(M^3)$ with 
\begin{align}
f\big|_{\partial E_k}\leq kL, \quad f\big|_{\partial E_{k+ 1}}> (k+ \frac{1}{2})L, \quad  \text{and} \ (k+ \frac{1}{2})L\ \text{is a regular value of} \ f. \nonumber 
\end{align}
Define $\mathcal{C}_k= \{x\in M^3: f(x)< (k+ \frac{1}{2})L\}\cap N_k$. Then $\mathcal{C}_k$ is a Caccioppoli set satisfying (\ref{assumption on Cacci set C}). 


From Lemma \ref{lem existence of smooth function} there is a function $\rho_k(x):\overline{N_k}\to [-L/2,L/2]$, which is continuous on $\overline{N_k}$ and smooth on $N_k$. Moreover $\rho_k(N_k)\subset(-L/2,L/2)$, and
\begin{align*}
    \rho_k|_{\partial E_k} = -L/2, \quad \rho_k|_{\partial E_{k+1}} = L/2, \quad |\nabla_M \rho_k|<(1+\varepsilon)^2.
\end{align*}

Now choose $h_k(x)= -\frac{2\pi(1+\varepsilon)}{L}\tan(\frac{\pi}{L}\rho_k(x))$ on $N_k$, from Proposition \ref{prop existence of bubbles}, there exists one $\mu$-bubble $\Omega_k$ with respect to $(h_k, 1, \mathcal{C}_k)$ in $(N_k, \partial_{-}N_k, \partial_{+}N_k)$. 

The choice of $h_k$ implies
\begin{align}
\frac{2\pi^2(1+\varepsilon)^2}{L^2}+ \frac{1}{2}h_k^2- |\nabla_M h_k|\geq 0. \label{ineq of h-1}
\end{align}

\textbf{Step (2)}. From Proposition \ref{prop connected property of annulus} and the assumption, there is $k_0\in \mathbb{Z}^+$, such that the sets $\overline{E_k}- E_{k+ 1}, \partial E_k$ are connected for any $k\geq k_0$. In the rest argument, we assume $k\geq k_0$.

Then there is a simple curve $\gamma:[0,1]\to\overline{E_k}- E_{k+ 1}$, with $\gamma(0)=q_1\in\partial E_k$, $\gamma(1)=q_2\in\partial E_{k+1}$ and $\gamma((0,1))\cap (\partial E_k\cup\partial E_{k+1}) =\varnothing$. Furthermore, form the density of transverse intersections, we may assume $\gamma$ intersects $\Omega_k$ and $\partial \Omega_k$ transversely. 

Let $I_2(\gamma, \partial \Omega_k)$ be the mod 2 intersection number of $\gamma$ and $\partial \Omega_k$. From the transversality, $\gamma\cap \Omega_k$ is an 1-dimensional manifold, which turns out to be a union of closed intervals. So $\displaystyle \#\partial \Big(\gamma\cap \Omega_k\Big) = 0 \mod 2$.  And then
\begin{align*}
    I_2(\gamma, \partial \Omega_k)= \#\gamma\cap \big(\partial \Omega_k\big) = \#\partial \Big(\gamma\cap\Omega_k\Big) = 0 \mod 2.
\end{align*}

From the definition of $\Omega_k$, we know $\partial E_k\subset \Omega_k$. Let $\displaystyle \partial\Omega_k-\partial E_k=\bigcup_i \Sigma_k^i$ be the disjoint union of connected components.  Since $\gamma\cap(\partial E_k)=\{q_1\}$, we get
\begin{align*}
     \sum_i I_2(\gamma,\Sigma_k^i) = \#\gamma\cap \big(\partial \Omega_k\big)-\#\gamma\cap (\partial E_k) = 1\mod 2.
\end{align*}

So there is $i_0$ such that
\begin{align*}
    I_2(\gamma,\Sigma_k^{i_0})=1.
\end{align*}
From the definition of $\Omega_k$, we know $\Omega_k\cap\partial E_{k+1}=\varnothing$. Thus $\Sigma_k^{i_0}\cap \partial E_{k+1}=\varnothing$. And we also have $\Sigma_k^{i_0}\cap \partial E_k=\varnothing$ from the definition.

\textbf{Step (3)}. We use $\Sigma_k$ to denote $\Sigma_k^{i_0}$, then $\Sigma_k$ satisfies the conclusion (1). If not, then there is a simple curve $\check{\gamma}:[0,1]\to\overline{E_k}- E_{k+ 1}$, such that $\check{\gamma}(0)=q_3\in\partial E_{k+1}$, $\check{\gamma}(1)=q_4\in\partial E_k$ and $\check{\gamma}((0,1))\cap (\partial E_k\cup\partial E_{k+1}) =\varnothing$. Moreover
\begin{align*}
    \check{\gamma}([0,1])\cap \Sigma_k=\varnothing.
\end{align*}
Since $\partial E_k$ and $\partial E_{k+1}$ are connected, there are two simple curves $\gamma_1:[0,1]\to\partial E_k$ and $\gamma_2:[0,1]\to\partial E_{k+1}$, such that $\gamma_1(0)=q_4$, $\gamma_1(1)=q_1$, $\gamma_2(0)=q_2$, $\gamma_2(1)=q_3$. Let $\tilde\gamma=\gamma*\gamma_2*\check{\gamma}*\gamma_1$ (see Figure \ref{figure: circle curve}), then
\begin{align}
   I_2(\tilde\gamma,\Sigma_k) = I_2(\gamma,\Sigma_k) + I_2(\gamma_2,\Sigma_k) + I_2(\check{\gamma},\Sigma_k) + I_2(\gamma_1,\Sigma_k)=1\mod 2.
\end{align}
On the other hand, from the assumption that $M$ is simply connected, $\tilde\gamma$ is homotopic equivalent to a constant curve at $q_1$. The mod 2 intersection number is invariant under the homotopic equivalence. So
\begin{align}
   I_2(\tilde\gamma,\Sigma_k) = I_2(q_1,\Sigma_k)=0 \mod 2,
\end{align}
which is a contradiction.

\begin{figure}[H]
    \centering
    \includegraphics[width=0.9\linewidth]{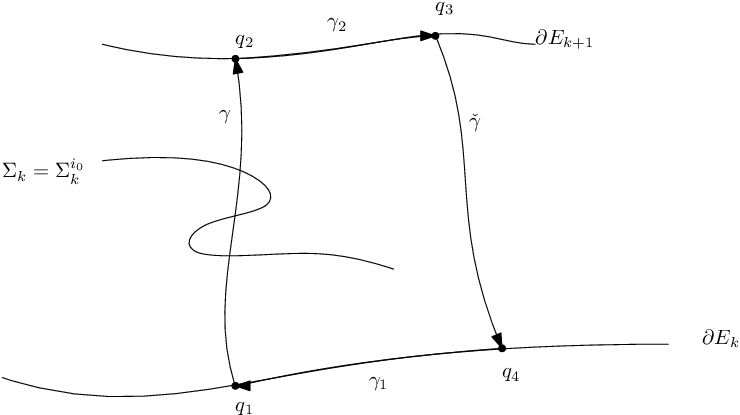}
    \caption{the closed curve $\tilde{\gamma}$}
    \label{figure: circle curve}
\end{figure}

\textbf{Step (4)}.  Now from Lemma \ref{lem 1st and 2nd vari of bubble}, for any $\psi\in C^\infty(\Sigma_k)$, we have 
\begin{align}
&\int_{\Sigma_k} |\nabla_{\Sigma_k} \psi|^2 - \frac{1}{2}(R_M- R_{\Sigma_k}+ |\mathring{A}_{\Sigma_k}|^2)\cdot \psi^2- \frac{1}{2}\{h_k^2+ 2\langle \nabla_M h_k, \vec{n}\rangle\}\psi^2  \geq 0. \nonumber 
\end{align}

Combining the assumption $R_M\geq 2$, we obtain
\begin{align}
&\int_{\Sigma_k} |\nabla_{\Sigma_k} \psi|^2 + (\frac{1}{2}R_{\Sigma_k}-1- \frac{1}{2}|\mathring{A}_{\Sigma_k}|^2)\cdot \psi^2- \frac{1}{2}\{h_k^2+ 2\langle \nabla_M h_k, \vec{n}\rangle\}\psi^2  \geq 0. \label{2nd vari ineq-1}
\end{align}

By (\ref{2nd vari ineq-1}) and (\ref{ineq of h-1}), there is
\begin{align}
&\int_{\Sigma_k} |\nabla_{\Sigma_k} \psi|^2 + (\frac{1}{2}R_{\Sigma_k}-1+ \frac{2\pi^2(1+\varepsilon)^2}{L^2})\cdot \psi^2 \label{inte ineq we can using F-S} \\ 
&\geq \int_{\Sigma_k} \{\frac{2\pi^2(1+\varepsilon)^2}{L^2}+ \frac{1}{2}h_k^2+ \langle \nabla_M h_k, \vec{n}\rangle\}\psi^2  \geq 0. \nonumber 
\end{align}

Let $\psi= 1$ in the above inequality, we get
\begin{align}
\int_{\Sigma_k} (1- \frac{2\pi^2(1+\varepsilon)^2}{L^2})\leq \frac{1}{2}\int_{\Sigma_k} R_{\Sigma_k}. \nonumber 
\end{align}

Note $\Sigma_k$ is a connected surface without boundary, by Gauss-Bonnet Theorem,
\begin{align}
\frac{1}{2}\int_{\Sigma_k} R_{\Sigma_k} = 2\pi\cdot \chi(\Sigma_k)\leq 4\pi.  \nonumber 
\end{align}

This implies $\displaystyle V(\Sigma_k)\leq \frac{4\pi}{1- \frac{2\pi^2(1+\varepsilon)^2}{L^2}}$. 

\textbf{Step (5)}.  Define $\displaystyle \mathcal{L}_k(f)\vcentcolon= -\Delta_{\Sigma_k} f+ (\frac{1}{2}R_{\Sigma_k}-1+ \frac{2\pi^2(1+\varepsilon)^2}{L^2})\cdot f$,  which is a self-dual operator.  From (\ref{inte ineq we can using F-S}),  we know that $\mathcal{L}_k$ is a non-negative operator.

Hence from the spectrum theory of self-dual operators,  there is a smooth function $\psi_k> 0$ defined on $\Sigma_k$ such that 
\begin{align}
-\Delta_{\Sigma_k} \psi_k+ (\frac{1}{2}R_{\Sigma_k}-1+ \frac{2\pi^2(1+\varepsilon)^2}{L^2})\cdot \psi_k \geq 0 . \nonumber 
\end{align}

From Lemma \ref{lem the diameter upper bound of 2-dim surface-SY}, we get $\displaystyle \mathrm{Diam}(\Sigma_k)\leq \frac{2\pi}{\sqrt{3(1- \frac{2\pi^2(1+ \varepsilon)^2}{L^2})}}$. 
}
\qed

\section{The height estimate of almost cylindrical region}\label{sec radius est for annulus}

The following lemma is closely related to \cite[Proposition $A. 4$]{CLS}; and our assumption is a little bit different from there,  we include its proof here for completeness.

We firstly recall the definition of length space (in some literature,  the length space defined below is also called complete length space).
\begin{definition}
    Let $(Y, d)$ be a metric space. For a path $\gamma:[a,b]\to Y$, let
    \begin{align*}
        \mathrm{length}(\gamma)=\sup\Big\{\sum_{i=0}^{n-1} d(\gamma(t_i),\gamma(t_{i+1})):a= t_0<t_1<\cdots<t_{n}=b\Big\}.
    \end{align*}
    For $x,y\in Y$, define $\displaystyle d'(x,y)=\inf\{\mathrm{length}(\gamma):\gamma\ \text{ is a path from}\ x \ \text{to}\ y\}$.  
    
If for any $x,y\in Y$, there is a path $\gamma_0$ from $x$ to $y$, such that $\mathrm{length}(\gamma_0)=d'(x,y)=d(x,y)<\infty$, then we call $(Y, d)$ a length space.
\end{definition}

\begin{lemma}\label{lem connected of some set}
{Suppose $(Y,  d)$ is a length space.  If $\displaystyle 0< r< \min\{\frac{1}{2}\mathrm{Diam}(Y, d),  L\}$,  then $B_{(y_0, 0)}(L)- B_{(y_0, 0)}(r)$ is path connected,  where $B_{(y_0,  0)}(r)$ is the open ball in $(Y\times \mathbb{R};  d\times d_{\mathbb{R}})$ centered at $(y_0, 0)\in Y\times \mathbb{R}$.  
}
\end{lemma}

\begin{remark}\label{rem counterexample}
    If $(Y,d)$ is merely a path connected metric space , the conclusion in Lemma \ref{lem connected of some set} does not hold.\footnote{The counterexample here is due to Jianqiao Shang in Université Paris-Saclay.}

    Let $Y=\mathbb{R}^2-(0,1)\times(-1,1)$ with the standard Euclidean metric, $y_0=(0,0)$ and $L=1.01$, then $B_{y_0}(L)$ is disconnected. Hence $B_{(y_0, 0)}(L)- B_{(y_0, 0)}(r)$ is disconnected (see Figure \ref{figure: counterexample for connectedness}).
\end{remark}

\begin{figure}[H]
    \centering
    \includegraphics[scale=0.8]{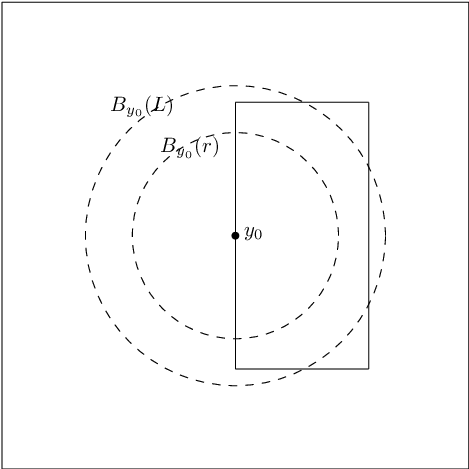}
    \caption{Counterexample when $(Y,d)$ is a path connected metric space}
    \label{figure: counterexample for connectedness}
\end{figure}

\begin{proof}
We first show
\begin{align*}
    X_+ \vcentcolon=(B_{(y_0, 0)}(L)- B_{(y_0, 0)}(r))\cap (Y\times \overline{\mathbb{R}^+})
\end{align*}
is path connected. For any $(y_1,t_1), (y_2,t_2)\in X_+$,  let $\eta(s)=(y(s),t(s))$ be a continuous curve from $(y_1,t_1)$ to $(y_2,t_2)$ in $B_{(y_0, 0)}(L)$,  where we use the fact that $Y$ is a length space.  
Replacing $t(s)$ by $\max\{t(s),0\}$, we may assume $t(s)\geq 0$.

Let 
\begin{align*}
    T(s)=\left\{\begin{array}{ll}
    \sqrt{r^2-d(y(s),y_0)^2}, & (y(s),t(s))\in B_{(y_0, 0)}(r)\\
    t(s), & \text{otherwise}
\end{array}\right..
\end{align*}
Then $T(s)\geq 0$ and $r^2\leq d(y(s),y_0)^2+T(s)^2<L^2$. So $(y(s),T(s))$ is a continuous curve in $X_+$ from $(y_1,t_1)$ to $(y_2,t_2)$. The same argument shows 
\begin{align*}
    X_-=(B_{(y_0, 0)}(L)- B_{(y_0, 0)}(r))\cap (Y\times \R_{\leq0})
\end{align*}
is also path connected.

For any $(y_1,t_1) \in X_+$, $ (y_2,t_2)\in X_-$, since $r< \frac{1}{2}\mathrm{Diam}(Y, d)$, there is a point $y'\in Y$ such that $r<d(y',y_0)<L$. Since $X_+$, $X_-$ are path connected, there is a path $\eta_1$ from $(y_1,t_1)$ to $(y',0)$, and a path $\eta_2$ from $(y_2,t_2)$ to $(y',0)$. Then $\eta_1*\eta_2$ is a path from $(y_1,t_1)$ to $(y_2,t_2)$. Hence $B_{(y_0, 0)}(L)- B_{(y_0, 0)}(r)=X_+\cup X_-$ is path connected.
\end{proof}

\begin{definition}\label{def tau nbhd of W}
{Assume $(X, d)$ is a metric space, for any set $W\subseteq X$ and $\tau> 0$, we define the $\tau$-neighborhood of $W$ as follows:
\begin{align}
U_\tau (W)\vcentcolon = \{x\in X: d(x, W)< \tau\}.  \nonumber 
\end{align} 
}
\end{definition}

\begin{lemma}\label{lem diam of partial geod spheres}
{Assume $(M^3, g)$ has only one end with $Rc\geq 0$ and $R\geq 2$. We also assume it is simply connected.  For any $p\in M^3$ and a unit speed geodesic ray $\gamma$ starting from $p$,  any $ \varepsilon\in (0,  1),  \delta\in (0,\frac{1}{100})$ and $L> 8\pi$,  there is $k_0= k_0(\delta, \varepsilon, L)\in \mathbb{Z}^+$ such that 
\begin{align}
\sup_{x, y\in \big(\partial B_p(r)\cap B_{\gamma(k\cdot L)}(5L)\big)}d_g(x, y)\leq \frac{4\pi}{\sqrt{3(1- \frac{2\pi^2(1+ \varepsilon)^2}{L^2})}}+ 9\delta\cdot L, \quad \quad \quad \forall k\geq k_0, r> 0.  \nonumber
\end{align}
}
\end{lemma}

\pf
{\textbf{Step (1)}.  In the rest of the argument,  we assume $\displaystyle c= \frac{2\pi}{\sqrt{3(1- \frac{2\pi^2(1+ \varepsilon)^2}{L^2})}}$.  

By Cheeger-Colding's almost splitting theorem \cite[Theorem $6.62$]{CC-Ann},  for any $\delta\in (0,  \frac{1}{100})$,  there is $k_1= k_1(\delta)\in \mathbb{Z}^+$ such that for all $L> 0$, the following holds:
\begin{align}
d_{GH}\Big(B_{\gamma(kL)}(5L)\subseteq (M^3, g),  B_{(y_0,  0)}(5L)\subseteq (Y\times \mathbb{R},  d\times d_{\mathbb{R}})\Big)\leq \delta\cdot L, \quad \quad \quad \forall k\geq k_1, \nonumber 
\end{align}
where $(Y, d)$ is a length space.

Furthermore,  the map $f: B_{\gamma(kL)}(5L)\rightarrow B_{(y_0,  0)}(5L)$ defined as follows is a $(\delta\cdot L)$-Gromov-Hausdorff approximation:
\begin{align}
f(x)= (\varphi(x),  d(p, x)- kL),  \nonumber 
\end{align}
where $\varphi: B_{\gamma(kL)}(5L)\rightarrow Y$ is a suitable map satisfying $\varphi(\gamma(t))= y_0$ for any $t\in (kL- 5L, kL+ 5L)$.  

For $\varepsilon\in (0, 1), L> \sqrt{2}\pi(1+ \varepsilon)$, from Proposition \ref{prop volume of separating bubble surfaces}; there is $k_2= k_2(\varepsilon, L)\in \mathbb{Z}^+$, and for each $k\geq k_2$, we can choose $\Sigma_k\in N_k$ as in Proposition \ref{prop volume of separating bubble surfaces}. 

We define $k_0\vcentcolon= k_0(\delta, \varepsilon, L)= k_1(\delta)+ k_2(\varepsilon, L)$, in the rest argument we always assume $k\geq k_0$.

Now we define
\begin{align}
t_k\vcentcolon= \min\{t> 0: \gamma(t)\in \Sigma_k\}.  \nonumber 
\end{align}

From the separating property of $\Sigma_k$,  we know that 
\begin{align}
(E_{k- 1}- \overline{E_{k+ 2}})- \Sigma_k= A_1\bigsqcup A_2, \label{seperation of A1 A2}
\end{align}
where $A_i$ is one of the two connected components of $(E_{k- 1}- \overline{E_{k+ 2}})- \Sigma_k$ and $i= 1, 2$ (see Figure \ref{fig: separtion of A1 A2}).

\begin{figure}[H]
    \centering
    \includegraphics[width=0.9\linewidth]{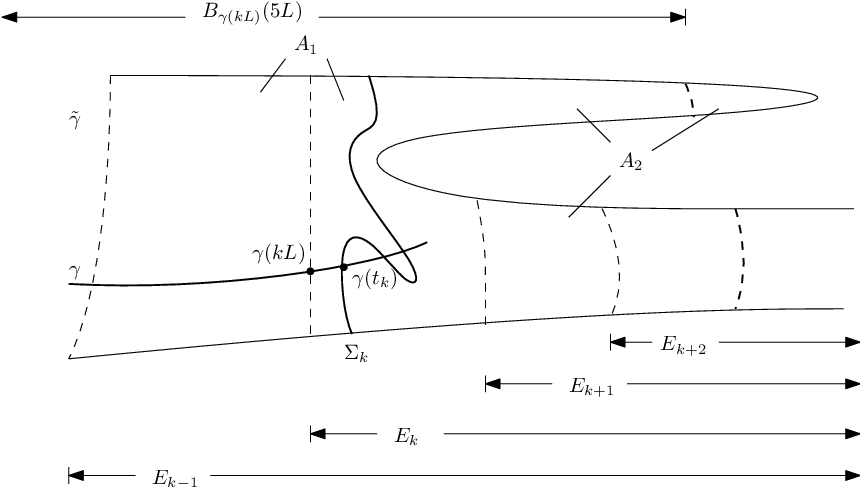}
    \caption{Picture for (\ref{seperation of A1 A2})}
    \label{fig: separtion of A1 A2}
\end{figure}

\textbf{Step (2)}.  Next we show
\begin{align}
\mathrm{Diam}(Y, d)\leq 2(c+ 4\delta \cdot L).  \label{bound of diameter of Y}
\end{align}

By contradiction,  assume (\ref{bound of diameter of Y}) does not hold. Note $\delta< \frac{1}{100}$ and $L> 8\pi$,  then we further have 
$\displaystyle c+ 4\delta \cdot L< \min\{\frac{1}{2}\mathrm{Diam}(Y, d),  L\}$. By Lemma \ref{lem connected of some set},  we obtain that $B_{z_1}(L)- B_{z_1}(c+ 4\delta L)$ is connected,  where $z_1\vcentcolon= (y_0,  t_k- kL)\in Y\times \mathbb{R}$.  

From Proposition \ref{prop volume of separating bubble surfaces}, we know that $\mathrm{Diam}(\Sigma_k)\leq c$. Define $\displaystyle S_k\vcentcolon= U_{\delta L}\Big(f\big[U_{2\delta L}(\Sigma_k)\big]\Big)$, then we have
\begin{align}
S_k\subseteq B_{z_1}(c+ 4\delta L).  \label{inclusion of Sk}
\end{align}

Now we define $\displaystyle \Lambda_i\vcentcolon = B_{z_1}(L)\cap \Big[U_{\delta\cdot L}(f(A_i))- S_k\Big]$ for $i= 1, 2$.  Note
\begin{align}
B_{z_1}(L)&\subseteq U_{\delta\cdot L}\Big(f(E_{k- 1}- \overline{E}_{k+ 2})\Big) \subseteq U_{\delta\cdot L}\Big(f(A_1)\Big)\bigcup U_{\delta\cdot L}\Big(f(A_2)\Big) \bigcup U_{\delta\cdot L}\Big(f(\Sigma_k)\Big). \nonumber 
\end{align}

Combining the above,  we obtain (see Figure \ref{fig: main set equa})
\begin{align}
B_{z_1}(L)= \Lambda_1\cup S_k\cup \Lambda_2. \label{main set equation}
\end{align}

\begin{figure}[H]
    \centering
    \includegraphics[scale=0.8]{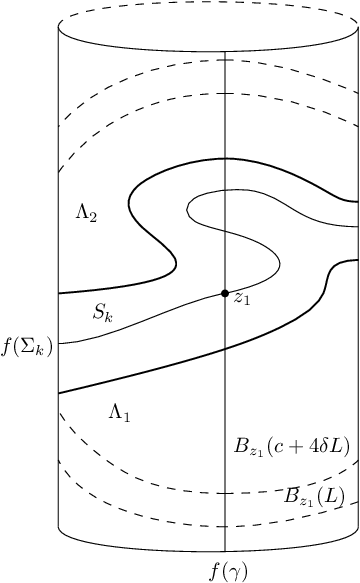}
    \caption{Picture for (\ref{main set equation})}
    \label{fig: main set equa}
\end{figure}

From (\ref{main set equation}) and (\ref{inclusion of Sk}),  we have
\begin{align}
\Big[B_{z_1}(L)- B_{z_1}(c+ 4\delta L)\Big]\subseteq \Lambda_1\cup \Lambda_2. \label{inclusion into two parts} 
\end{align}

Now for any $q_i\in \Lambda_i- S_k$,  there are $\tilde{q}_i\in A_i- U_{2\delta L}(\Sigma_k)$ such that $d(q_i, f(\tilde{q}_i))< \delta\cdot L$.  Therefore
\begin{align}
d(q_1, q_2)&\geq d(f(\tilde{q}_1), f(\tilde{q}_2))- 2\delta L\geq d(\tilde{q}_1, \tilde{q}_2)- 3\delta L \nonumber \\
&\geq d(\tilde{q}_1, \Sigma_k)+ d(\tilde{q}_2, \Sigma_k)- 3\delta L\geq \delta L> 0.   \label{disjoint} 
\end{align}

By (\ref{inclusion of Sk}) and (\ref{disjoint}), we get that 
\begin{align}
d\Big(\Lambda_1\cap \Big[B_{z_1}(L)- B_{z_1}(c+ 4\delta L)\Big], \Lambda_2\cap \Big[B_{z_1}(L)- B_{z_1}(c+ 4\delta L)\Big]\Big)\geq \delta L. \label{dist between two subsets}
\end{align}

Now from (\ref{inclusion into two parts}), we get that 
\begin{align*}
&\Big\{\Lambda_1\cap \Big[B_{z_1}(L)- B_{z_1}(c+ 4\delta L)\Big]\Big\}\cup \Big\{\Lambda_2\cap \Big[B_{z_1}(L)- B_{z_1}(c+ 4\delta L)\Big]\Big\}\\
&= \Big[B_{z_1}(L)- B_{z_1}(c+ 4\delta L)\Big].
\end{align*}
Combining (\ref{dist between two subsets}), we get that $\Big[B_{z_1}(L)- B_{z_1}(c+ 4\delta L)\Big]$ is not connected,  which leads to a  contradiction.  We complete the proof of (\ref{bound of diameter of Y}).

By the definition of the $(\delta\cdot L)$-Gromov-Hausdorff approximation map $f$ and (\ref{bound of diameter of Y}), for any $x, y\in \big(\partial B_p(r)\cap B_{\gamma(k\cdot L)}(5L)\big)$, we get
\begin{align}
d_g(x, y)&\leq d(f(x), f(y))+ \delta\cdot L= d((\varphi(x), r- kL), (\varphi(y), r- kL))+ \delta\cdot L\nonumber \\
&= d_Y(\varphi(x), \varphi(y))+ \delta\cdot L\leq 2c+ 9\delta\cdot L. \nonumber  
\end{align}
}
\qed

The main result of this section is Lemma \ref{lem annulus is not too far from p},  which is an improvement of \cite[Claim $3$,  page $5$]{CLS}.  In fact, this improvement is asymptotically sharp,  which will be used in the proof of Theorem \ref{thm sharp volume growth}.  

\begin{lemma}\label{lem annulus is not too far from p}
{Assume $(M^3, g)$ has only one end and $Rc\geq 0, R\geq 2$, furthermore assume it is simply connected.  For any $\varepsilon\in (0, 1), \delta\in (0, \frac{1}{100}), L> 8\pi$, there is $k_0= k_0(\varepsilon, \delta, L)\in \mathbb{Z}^+$; such that (see Figure \ref{fig:lem annulus is not too far from p})
\begin{align}
N_k\subseteq \Big[B_p((k+ 1)L+ 2c+ 10\delta\cdot L)- \overline{B_p(kL)}\Big],  \quad \quad \quad \forall k\geq k_0,   \label{annulus inclusion relation}
\end{align}
where $\displaystyle c= \frac{2\pi}{\sqrt{3(1- \frac{2\pi^2(1+ \varepsilon)^2}{L^2})}}$. 
}
\end{lemma}

\begin{figure}[H]
    \centering
    \includegraphics[width=\linewidth]{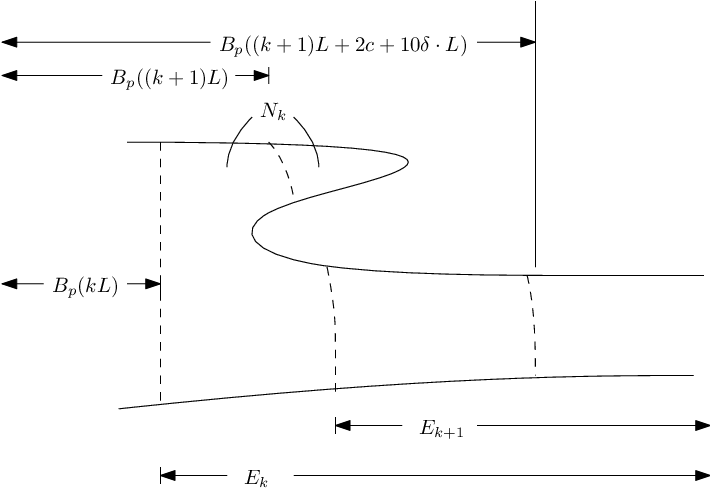}
    \caption{Picture for (\ref{annulus inclusion relation})}
    \label{fig:lem annulus is not too far from p}
\end{figure}

\pf
{\textbf{Step (1)}. Choose $k_0$ from Lemma \ref{lem diam of partial geod spheres}.   

By contradiction.  Otherwise,  there is $x'\in  N_k- B_p((k+ 1)L+ 2c+ 10\delta\cdot L)$, where $L> 8\pi$ and $k\geq k_0$ is some integer.  Let $\tilde{\gamma}$ be the geodesic segment $\overline{px'}$,  and define $x''= \tilde{\gamma}\cap \partial B_p((k+ 1)L+ 2c+ 10\delta\cdot L)$.  

Choose a ray $\gamma$ with $\gamma(0)= p$ and $\gamma\big|_{(kL,\infty)}\subseteq E_k$.  By the choice of $L$, we have $c\leq \frac{1}{4}L$.

By Proposition \ref{prop volume of separating bubble surfaces}, we can find one $\mu$-bubble surface $\Sigma_{k- 1}\subseteq N_{k- 1}$ with $\mathrm{Diam}(\Sigma_{k- 1})\leq c$, which separates $\partial E_{k- 1}$ from $\partial E_k$. So we can take $x_{k- 1}\in (\tilde{\gamma}\cap \Sigma_{k- 1}), y_{k- 1}\in (\gamma\cap \Sigma_{k- 1})$. Then we have (see Figure \ref{fig:proof in lem annulus}.)
\begin{align}
d(x'', \gamma(kL))&\leq d(x'',  x_{k- 1})+ d(x_{k- 1},  y_{k- 1})+ d(y_{k- 1},  \gamma(kL))\leq (2L+ 2c+ 10\delta L)+ c+ L  \nonumber \\
&\leq 5L.   \label{metric ineq for points}
\end{align}

Therefore $x''\in B_{\gamma(kL)}(5L)$.

\begin{figure}[H]
    \centering
    \includegraphics[width=\linewidth]{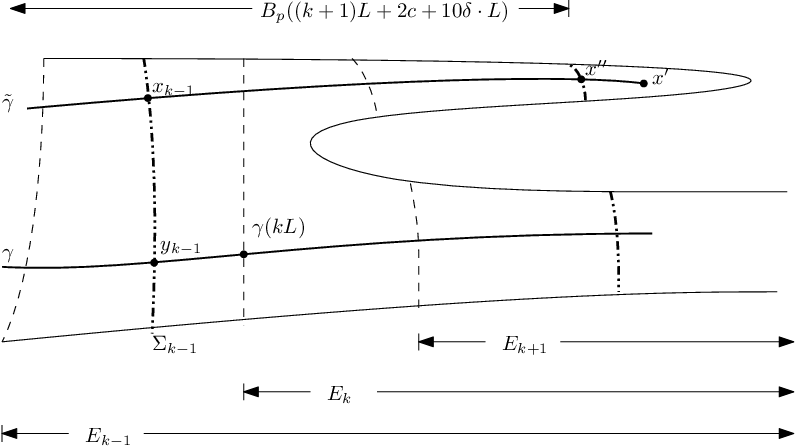}
    \caption{Picture for (\ref{metric ineq for points})}
    \label{fig:proof in lem annulus}
\end{figure}

\textbf{Step (2)}.  Note $x'', \gamma((k+ 1+ 10\delta)L+ 2c)\in \partial B_p((k+ 1+ 10\delta)L+ 2c)\cap B_{\gamma(kL)}(5L)$,  from Lemma \ref{lem diam of partial geod spheres},  we have
\begin{align}
d(x'',  \gamma((k+ 1+ 10\delta)L+ 2c))\leq  2c+ 9\delta\cdot L.  \label{x'' and gamma is close}
\end{align}

On the other hand,  because $x''\in \overline{N_k}$,  the set $\partial B_p((k+ 1)L)$ separates $x''$ from $\gamma$.  From (\ref{x'' and gamma is close}),  we have
\begin{align}
2c+ 9\delta\cdot L\geq d(x'',  \gamma)\geq d(x'',  \partial B_p((k+ 1)L))= 2c+ 10\delta\cdot L.  \nonumber 
\end{align}
This is the contradiction.
}
\qed

\section{The volume comparison of hypersurfaces}\label{sec area compa}

For any $q\in M^n$,  we can write $q= (r, \theta)$ in terms of polar normal coordinates at $p$, where $r= d(p, q)$.  And it is well known that we can write the volume element of $(M^n,  g)$ as 
\begin{align}
d\mu_g= J(t, \theta)dtd\theta,  \nonumber 
\end{align}
where $d\theta$ is the area element of the unit $(n- 1)$-sphere $\mathbb{S}^{n- 1}$. The area element of geodesic sphere $\partial B_p(t)$ is given by $J(t, \theta)d\theta$.



\begin{lemma}\label{lem det of positive semi definite matrix}
    Let $X, Y$ be two positive semi-definite $(n\times n)$-matrices, then
    \begin{align*}
        \det(X+Y)\geq \det(X)+\det(Y).
    \end{align*}
\end{lemma}
\begin{proof} The inequality is trivial when $\det(X)+\det(Y)=0$. Without loss of generality,  we assume $\det(X)>0$ in the following proof.
\begin{align*}
    \det(X+Y)=&\det(\sqrt{X}(I+\sqrt{X}^{-1}Y\sqrt{X}^{-1})\sqrt{X})=\det(X)\det(I+\sqrt{X}^{-1}Y\sqrt{X}^{-1}).
\end{align*}
Let $\lambda_i\geq 0$ be the eigenvalues of $\sqrt{X}^{-1}Y\sqrt{X}^{-1}$, then 
\begin{align*}
    \det(I+\sqrt{X}^{-1}Y\sqrt{X}^{-1})=& \prod_{i=1}^n(1+\lambda_i)\geq1+\prod_{i=1}^n\lambda_i\\
    =& 1+\det(\sqrt{X}^{-1}Y\sqrt{X}^{-1})=1+\det(Y)\det(X)^{-1}.
\end{align*}
We complete the proof.
\end{proof}

For $p\in M$,  we define $\mathrm{Cut}(p)$ as the set of all cut points of $p$.

\begin{lemma}\label{lem volume of surface in term of polar Jac}
  If $(M^n, g)$ is a complete Riemannian manifold and $r: \Omega\rightarrow \mathbb{R}^+$ is a smooth function, where $\Omega\subseteq \mathbb{S}_pM$ is an open set. Let $\varphi(\theta)=\exp_p(r(\theta)\theta)$; then
    \begin{align}
        J_\varphi(\theta)\geq J(r(\theta),\theta), \quad \quad \quad \forall \theta\in \Omega; \nonumber 
    \end{align}
where $J_\varphi(\theta)$ is the Jacobian determinant of $\varphi$. 

\end{lemma}

\begin{proof}[Proof:] Assume the local normal coordinates of $\mathbb{S}_pM$ are $\theta^1,\theta^2,\cdots,\theta^{n-1}$, then
\begin{align*}
        J_\varphi(\theta)=\sqrt{\det\Bigg[\Bigg(\Big<\frac{\partial \varphi(\theta)}{\partial \theta^i},\frac{\partial \varphi(\theta)}{\partial \theta^j}\Big>_q\Bigg)_{1\leq i,j\leq n-1}\Bigg]}.
    \end{align*}
where $\bracket{\cdot,  \cdot}_q$ is the inner product with respect to the metric on $T_{q}M$ and $q= \varphi(\theta)$.

Note
    \begin{align*}
        \frac{\partial \varphi(\theta)}{\partial \theta^i}=&\d\exp_p\big(r(\theta)\theta\big)\left(\frac{\partial \big(r(\theta)\theta\big)}{\partial \theta^i}\right) =\d\exp_p\big(r(\theta)\theta\big)\bigg(\frac{\partial r(\theta)}{\partial \theta^i}\theta + r(\theta)\frac{\partial\theta}{\partial \theta^i}\bigg)\\
        =&\frac{\partial r(\theta)}{\partial \theta^i}\d\exp_p\big(r(\theta)\theta\big)(\theta) + r(\theta)\d\exp_p\big(r(\theta)\theta\big)\bigg(\frac{\partial\theta}{\partial \theta^i}\bigg).
    \end{align*} 
    
    For simplicity, let
    \begin{align*}
        r_i=\frac{\partial r(\theta)}{\partial \theta^i},\quad \quad  A= \d\exp_p\big(r(\theta)\theta\big)(\theta), \quad \quad 
        B_i=r(\theta)\d\exp_p\big(r(\theta)\theta\big)\bigg(\frac{\partial\theta}{\partial \theta^i}\bigg).
    \end{align*}
We have $\displaystyle  \frac{\partial \varphi(\theta)}{\partial \theta^i}=r_iA+B_i$. 

    From the definition, we know $\bracket{\theta,\theta}_p=1$. So
    \begin{align*}
        \bracket{\theta,\frac{\partial \theta}{\partial \theta^i}}_p=0, \quad i = 1,2,\cdots,n-1.
    \end{align*}
    From Gauss lemma, 
    \begin{align*}
        \bracket{A,B_i}_q=&r(\theta)\bracket{\d\exp_p\big(r(\theta)\theta\big)(\theta),\d\exp_p(r(\theta)\theta)\left(\frac{\partial \theta}{\partial \theta^i}\right)}_q        =r(\theta)\bracket{\theta,\frac{\partial \theta}{\partial \theta^i}}_p=0.
    \end{align*}

 From the above we get 
    \begin{align*}
        J_\varphi(\theta)^2=\det\Big[\Big(r_ir_j |A|_q^2+\bracket{B_i,B_j}_q\Big)_{1\leq i,j\leq n-1}\Big].
    \end{align*}

For fixed $t> 0$, define $\displaystyle \psi_t(\theta)=\exp_p(t\theta): \mathbb{S}_p(M)\rightarrow M^n$,  then
    \begin{align*}
        \left.\frac{\partial \psi_t(\theta)}{\partial \theta^i}\right|_{t=r(\theta)}
       &=\left.\d\exp_p(t\theta)\left(t\frac{\partial \theta}{\partial \theta^i}\right)\right|_{t=r(\theta)}        =r(\theta)\d\exp_p\big(r(\theta)\theta\big)\bigg(\frac{\partial\theta}{\partial \theta^i}\bigg)=B_i.
    \end{align*} 
    And the Jacobian of $J_{\psi_t}$ is as follows:
    \begin{align*}
        J_{\psi_t}(\theta)|_{t=r(\theta)}=&\sqrt{\left.\det\Bigg[\Bigg(\Big<\frac{\partial \psi_t(\theta)}{\partial \theta^i},\frac{\partial \psi_t(\theta)}{\partial \theta^j}\Big>_q\Bigg)_{1\leq i,j\leq n-1}\Bigg]\right|_{t=r(\theta)}}\\
        =&\sqrt{\det\Big[\Big(\bracket{B_i,B_j}_q\Big)_{1\leq i,j\leq n-1}\Big]}.
    \end{align*}

Since $\theta^1,\theta^2,\cdots,\theta^{n-1}$ are normal coordinates, 
\begin{align*}
    \bracket{\frac{\partial \theta}{\partial \theta^i}, \frac{\partial \theta}{\partial \theta^j}}=&\delta_{ij},\quad 1\leq i, j\leq n-1. 
\end{align*}

Combining $\bracket{\theta,\theta}_p=1$ and $\bracket{\theta,\frac{\partial \theta}{\partial \theta^i}}_p=0$, we know that $\theta,\frac{\partial \theta}{\partial \theta^1}, \frac{\partial \theta}{\partial \theta^2},\cdots,\frac{\partial \theta}{\partial \theta^{n-1}}$ form an orthonormal basis of $T_pM$.  So $J(t,\theta)=J_{\psi_t}(\theta)$. From Lemma \ref{lem det of positive semi definite matrix}, we have
    \begin{align*}
        &J_\varphi(\theta)^2-J(r(\theta),\theta)^2\\
        =&\det\Big[\Big(r_ir_j |A|_q^2+\bracket{B_i,B_j}_q\Big)_{1\leq i,j\leq n-1}\Big]-\det\Big[\Big(\bracket{B_i,B_j}_q\Big)_{1\leq i,j\leq n-1}\Big]\geq 0.
    \end{align*}
    
The conclusion follows.
\end{proof}

\begin{lemma}\label{lem surface area comparison}
If $(M^n, g)$ has $Rc\geq 0$ and only one end,   suppose $\Sigma \subset  N_{k- 1}$ is a smooth, compact hypersurface separating $\partial E_{k- 1}$ from $\partial E_{k}$, then
\begin{align}
V\Big((\partial B_p(s)\cap N_k)- \mathrm{Cut}(p)\Big)\leq (\frac{s}{(k- 1)L})^{n- 1}\cdot V(\Sigma), \quad \quad \forall s> kL. \nonumber 
\end{align} 
\end{lemma}

\begin{proof}[Proof:] Let $\tilde{\Sigma}\vcentcolon= \Sigma-\mathrm{Cut}(p)$, define $\Pi:\tilde{\Sigma}\to \mathbb{S}_p(M)$ as $\displaystyle \Pi(q)=\frac{\exp_p^{-1}(q)}{|\exp_p^{-1}(q)|_{g(p)}}$,  where $\mathbb{S}_pM$ is the unit tangent space of $M$ at $p$. From the definition of cut points,  the map $\Pi$ is a well-defined smooth map.

From the separating property of $\Sigma$, we know that 
\begin{align}
\Theta_{k}\vcentcolon=\frac{1}{s}\exp_p^{-1}\big((\partial B_p(s)\cap N_k)-\mathrm{Cut}(p)\big)\subseteq \Pi(\tilde{\Sigma}). \nonumber 
\end{align}
Applying Sard's theorem on the map $\Pi:\tilde{\Sigma}\to \mathbb{S}_p(M)$, we obtain that the set of critical values of $\Pi$ (denoted as $\mathcal{C}(\Pi)$) satisfies $\mathcal{H}^{n- 1}(\mathcal{C}(\Pi))= 0$, where $\mathcal{H}^{n- 1}$ is $(n- 1)$-dim Hausdorff measure on $\mathbb{S}_pM$.

Define
\begin{align}
\Theta_{k}'= \Theta_{k}- \mathcal{C}(\Pi). \label{0 diff between sets}
\end{align}
Let $\Sigma'\subset (\tilde\Sigma\cap\Pi^{-1}(\Theta_{k}))$ be the regular points of $\Pi$ with image in $\Theta_{k}$. Then $\Sigma'$ is an open subset of $\S_p M$ from the implicit function theorem, and 
\begin{align}
\Theta_{k}'\subset\Pi(\Sigma'). \nonumber 
\end{align}

For any $q_0\in\Sigma'$, there is a local coordinate chart $(U,h)$ of $\Sigma'$ containing $q_0$, such that $\Pi|_{U}$ is a diffeomorphism between $U$ and $\Pi(U)$. And 
\begin{align}
\varphi(\theta)\vcentcolon=\exp_p(r(\theta)\cdot \theta)= (\Pi\big|_U)^{-1}(\theta), \nonumber 
\end{align}
is a smooth function defined on $\Pi(U)\subseteq \mathbb{S}_pM$. 

For any $\theta\in \Theta_{k}'$, note $(k- 1)L\leq r(\theta)\leq kL< s$. From Lemma \ref{lem volume of surface in term of polar Jac} and Bishop-Gromov's volume comparison Theorem, we get
\begin{align}
    J_{(\Pi\big|_U)^{-1}}(\theta)&= J_\varphi(\theta) \geq J(r(\theta),\theta) \geq \bigg(\frac{r(\theta)}{s}\bigg)^{n- 1}J(s,\theta) \geq \bigg(\frac{(k- 1)L}{s}\bigg)^{n- 1}J(s,\theta). \nonumber 
\end{align}

Finally by the area formula, (\ref{0 diff between sets}) and $\mathcal{H}^{n- 1}(\mathcal{C}(\Pi))= 0$, we get
\begin{align*}
    V(\Sigma)\geq & V(\Sigma')\geq V(\Pi^{-1}(\Theta_{k}'))
    \geq \bigg(\frac{(k- 1)L}{s}\bigg)^{n- 1}\int_{\Theta_{k}'} J(s,\theta)\d \mathcal{H}^{{n- 1}}(\theta)\\
    &=\bigg(\frac{(k- 1)L}{s}\bigg)^{n- 1}\int_{\Theta_{k}} J(s,\theta)\d \mathcal{H}^{{n- 1}}(\theta)\\
&=\bigg(\frac{(k- 1)L}{s}\bigg)^{n- 1} V\Big((\partial B_p(s)\cap N_k)- \mathrm{Cut}(p)\Big).
\end{align*}
\end{proof}

\section{The sharp linear volume growth}\label{sec main proof}

\begin{theorem}\label{thm sharp volume growth} 
{For any complete non-compact Riemannian manifold $(M^3, g)$ with $Rc\geq 0$ and $R\geq 2$, 
\begin{enumerate}
\item[(1)].  if $(M^3, g)$ has only one end,  then $\displaystyle 0< \mathrm{V}_{M^3, 1}\leq 2\pi$,
\item[(2)]. if $(M^3,g)$ has at least two ends, then $(M^3,g)$ is isometric to a cylinder $S\times\mathbb{R}$ with $S$ being a closed surface of sectional curvature at least one, and $\displaystyle 0< \mathrm{V}_{M^3, 1}\leq 4\pi$. 
\item[(3)].  moreover $\mathrm{V}_{M^3, 1}= 4\pi$ if and only if $(M^3,g)$ is isometric to a cylinder $\mathbb{S}^2\times\mathbb{R}$. 
\end{enumerate}
}
\end{theorem}

\pf
{\textbf{Step (1)}. Note any complete manifold with nonnegative Ricci curvature has at least linear volume growth (see \cite{Yau-linear}). We know that $\displaystyle \mathrm{V}_{M^3, 1}> 0$ generally. 

If $(M^3, g)$ has more than one ends, then there is at least one geodesic line in $M^3$.
From Cheeger-Gromoll's splitting Theorem, we know that $M^3$ is isometric to $N^2\times \mathbb{R}$. Using $R\big|_M\geq 2$, we get that $R\big|_N\geq 2$. From Bonnet-Myers' Theorem and the Bishop-Gromov's volume comparison Theorem, we get that $N^2$ is compact and $V(N^2)\leq V(\mathbb{S}^2)= 4\pi$. The direct computation yields 
\begin{align}
\mathrm{V}_{M^3, 1}= \varlimsup\limits_{r\rightarrow\infty}\frac{V(B_p(r))}{2r}\leq \varlimsup\limits_{r\rightarrow\infty}\frac{V(N^2)\times (2r)}{2r}\leq 4\pi. \label{splitting ineq}
\end{align}

The rigidity part of (\ref{splitting ineq}) follows from the rigidity part of the Bishop-Gromov's volume comparison Theorem for $N^2$. 

\textbf{Step (2)}. In the rest argument, we assume that $(M^3, g)$ has only one end. We firstly consider the case $\pi_1(M^3)= 0$.

For any $\varepsilon\in (0, 1), \delta\in (0,  \frac{1}{100})$,  choose $L>  8\pi$.  Choose $k_0= k_0(\varepsilon, \delta,  L)\in \mathbb{Z}^+$,  such that for $k\geq k_0$,  the conclusions in Proposition \ref{prop volume of separating bubble surfaces} and Lemma \ref{lem annulus is not too far from p} both hold.   

In the rest argument, unless otherwise mentioned, all $k\geq k_0$.

Now from Proposition \ref{prop volume of separating bubble surfaces}, there is a compact smooth surface $\Sigma_k\subseteq N_k$ separating $\partial E_k$ from $\partial E_{k+ 1}$, and
\begin{align}
V(\Sigma_{k})\leq \frac{4\pi}{1- \frac{2\pi^2(1+\varepsilon)^2}{L^2}}. \label{upper bound of are of Sigma}
\end{align}

From Lemma \ref{lem surface area comparison} and (\ref{upper bound of are of Sigma}),  for any $r\in (kL,  (k+ 1+ 10\delta)\cdot L+ 2c)$,  we have
\begin{align}
V\Big((\partial B_p(r)\cap N_k)- \mathrm{Cut}(p)\Big)\leq \left(\frac{r}{(k- 1)L}\right)^2\cdot V(\Sigma_{k- 1})\leq \left(\frac{r}{(k- 1)L}\right)^2\cdot \frac{4\pi}{1- \frac{2\pi^2(1+\varepsilon)^2}{L^2}},\label{slice upper}
\end{align}
where $\displaystyle c= \frac{2\pi}{\sqrt{3(1- \frac{2\pi^2(1+ \varepsilon)^2}{L^2})}}$.

Note $V(\mathrm{Cut}(p))= 0$, where $V$ is the measure with respect to metric $g$.  Now using (\ref{slice upper}) and Lemma \ref{lem annulus is not too far from p},  we get
\begin{align}
V(N_k)&= V(N_k- \mathrm{Cut}(p))\leq \int_{kL}^{(k+ 1+ 10\delta)\cdot L+ 2c} V\Big((\partial B_p(r)\cap N_k)- \mathrm{Cut}(p)\Big) dr  \nonumber \\
&\leq  \Big(\frac{k^2+ (1+ 10\delta+ \frac{2c}{L})k+ 3^{-1}(1+ 10\delta+ \frac{2c}{L})^2}{(k- 1)^2}\Big)\cdot \frac{4\pi\cdot [(1+ 10\delta)L+ 2c]}{1- \frac{2\pi^2(1+\varepsilon)^2}{L^2}}. \nonumber 
\end{align}

Therefore for $r\in (kL, (k+ 1)L]$,  using $c\leq \frac{1}{4}L$,  there is 
\begin{align}
V(B_p(r))&\leq V(M- E_{k_0})+ \sum_{j= k_0}^{k}V(N_j) \nonumber \\
&\leq C+ \frac{4\pi\cdot [(1+ 10\delta)L+ 2c]}{1- \frac{2\pi^2(1+\varepsilon)^2}{L^2}}\cdot \sum_{j= k_0}^{k}\Big(\frac{j^2+ (1+ 10\delta + \frac{2c}{L})j+ 3^{-1}(1+ 10\delta+ 4^{-1})^2}{(j- 1)^2}\Big) \nonumber \\
& \leq C+ \frac{4\pi\cdot [(1+ 10\delta)L+ 2c]}{1- \frac{2\pi^2(1+\varepsilon)^2}{L^2}}\cdot  \sum_{j= k_0}^{k}\Big(\frac{j^2+ (1+ 10\delta+ \frac{2c}{L})j}{(j- 1)^2}\Big), \nonumber 
\end{align}
where $C> 0$ is some universal constant independent of $k$,  and $C$ may depend on $L$ and change from line to line.

Direct computation yields
\begin{align}
\varlimsup_{r\rightarrow\infty} \frac{V(B(r))}{r}\leq \varlimsup_{k\rightarrow\infty}\frac{C+ \frac{4\pi\cdot [(1+ 10\delta)L+ 2c]}{1- \frac{2\pi^2(1+\varepsilon)^2}{L^2}}\cdot \sum_{j= k_0}^{k}\Big(\frac{j^2+ (1+ 10\delta+ \frac{2c}{L})j}{(j- 1)^2}\Big)}{kL}\leq \frac{4\pi\cdot (1+ 10\delta+ \frac{2c}{L})}{1- \frac{2\pi^2(1+\varepsilon)^2}{L^2}} . \label{final ineq with L and delta} 
\end{align}

Note $\displaystyle \lim_{L\rightarrow\infty}\frac{c}{L}= 0$,  let $L\rightarrow\infty,  \delta\rightarrow 0$ in (\ref{final ineq with L and delta}), we get $\displaystyle  \mathrm{V}_{M^3, 1}\leq 2\pi$.

\textbf{Step (3)}. If the fundamental group $\pi_1(M^3)\neq \{1\}$, assume $(\tilde{M}^3, \tilde{g})$ is the universal cover of $(M^3, g)$. Because the number of elements in $\pi_1(M^3)$ satisfies $\#\pi_1(M^3)\geq 2$, we get $\displaystyle  \mathrm{V}_{M^3, 1}\leq  \frac{1}{2}\mathrm{V}_{\tilde{M}^3, 1}$.

Note $(\tilde{M}^3, \tilde{g})$ is simply connected with $Rc\geq 0, R\geq 2$. From Step (1) and Step (2), we get that $\displaystyle \mathrm{V}_{\tilde{M}^3, 1}\leq 4\pi$. Therefore we obtain that $\displaystyle  \mathrm{V}_{M^3, 1}\leq 2\pi$ in this case. 
}
\qed

In the end of this article,  we construct examples to show that only $\displaystyle  \mathrm{V}_{M^3, 1}=4\pi$  case yields the rigidity.

\begin{remark}\label{rem two end models}
{When $(M^3,g)$ has more than one ends, the model space is as in Figure \ref{figure: Splitting-3}. 

Let $(M^3, g)=\S^2(r_0)\times \mathbb{R}$, where $0<r_0<1$; then we have $V_{M^3,1}=4\pi r_0^2\in (0, 4\pi)$. On the other hand, let $E(a)$ be the ellipsoid 
\begin{align*}
    E(a)=\{(x,y,z)\in\R^3:\frac{x^2}{a^2}+y^2+z^2=r_0^2\},
\end{align*}
where $1<a\leq r_0^{-1}$. Let $\displaystyle b= \frac{V(E(a))}{V(\S^2(r_0))}>1$ and $N(a)=\frac{1}{\sqrt{b}}E(a)$. Then $Rc(N(a))\geq 0$, $R(N(a))\geq \frac{2b}{a^2r_0^2}\geq 2$, and $V(N(a))=4\pi r_0^2$. We also have $V_{M^3,1}=4\pi r_0^2$ where $(M^3, g)= N(a)\times \mathbb{R}$. 

Hence there is no rigidity for the case: $V_{M^3,1}\in (0, 4\pi)$ and $(M^3, g)$ has more than one ends.

\begin{figure}[H]
    \centering
    \includegraphics{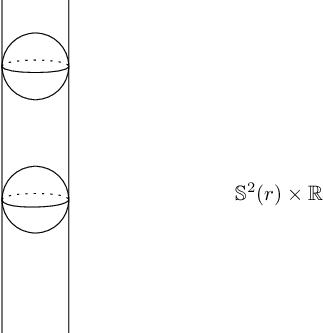}
    \caption{The model space with two ends}
    \label{figure: Splitting-3}
\end{figure}
}
\end{remark}

\begin{remark}\label{rem one end models}
{When $(M^3,g)$ has only one end, the model space is as in Figure \ref{figure: The equality model in Cohn-Vossen's inequality-half-cylinder-3}.  

We construct examples as follows, which show that there is no any rigidity result for any $\mathrm{V}_{M^3, 1}\in (0, 2\pi]$.

Let $(M, g)= (\R^3,  dr^2+ f(r)^2d\mathbb{S}^2)$, and $\theta$ is coordinate on $\mathbb{S}^2$.  Simple computations show that 
\begin{align*}
&Rc(\partial r, \partial r)= -2\frac{f''}{f}, \quad \quad \quad Rc(\partial\theta, \partial\theta)= \frac{1- (f')^2}{f^2}- \frac{f''}{f},\\
&R= 2\frac{1- (f')^2}{f^2}- 4\frac{f''}{f}. 
\end{align*}
We want to construct function $f: [0, \infty)\rightarrow \overline{\mathbb{R}^+}$ satisfies
\begin{equation}\label{Rc condi for f}
Rc\geq0,\quad\quad\quad R\geq2,\end{equation}
and
\begin{align}
f^{(\text{even})}(0)= 0, \quad \quad \quad f'(0)= 1.  \label{initial condition for f}
\end{align}

Direct computation yields
\begin{align}
\mathrm{V}_{M^3, 1}= \varlimsup_{r\rightarrow\infty}\frac{V(B_p(r))}{2r}= \varlimsup_{r\rightarrow\infty}\frac{1}{2r}\int_0^r f^2(s)\cdot 4\pi ds= 2\pi\varlimsup_{r\rightarrow\infty}\frac{\int_0^r f^2(s)ds}{r}. \nonumber 
\end{align}

Now, fix $0<b\leq 1$. 

\begin{enumerate}
\item[Example $(1)$].  Let 
$ f(r)=f_1(r)\vcentcolon=\frac{2b}{\pi}\arctan(\frac{\pi}{2b}r)$. Note $f_1$ is a smooth odd function on $\R$, so $f_1^{(\text{even})}(0)= 0$. Furthermore,
$$  f_1'(r)=\frac{1}{1+(\frac{\pi r}{2b})^2},\quad f_1''(r)=-\frac{\pi^2 r}{2b^2(1+(\frac{\pi r}{2b})^2)^2}\leq 0 \quad \text{when } r\geq 0.$$
Direct computation yields 
$$1-f_1'(r)^2-f_1(r)^2\geq 0.$$ Hence, $f_1$ satisfies (\ref{Rc condi for f}) and (\ref{initial condition for f}).
\item[Example $(2)$].   Let 
$$f(r)=f_2(r)\vcentcolon =b\tanh(\frac{r}{b}).$$  
Then
\begin{align*}
    f_2'(r)=\frac{4}{\left(e^{-r/b}+e^{r/b}\right)^2},\quad f_2''(r)=\frac{8 \left(e^{-r/b}-e^{r/b}\right)}{b \left(e^{-r/b}+e^{r/b}\right)^3}\leq 0\quad \text{when } r\geq 0.
\end{align*}
Furthermore,
\begin{align*}
    1-f_2'(r)^2-f_2(r)^2=\tanh^2(\frac{r}{b})\left(1-b^2+\frac{4}{\left(e^{-r/b}+e^{r/b}\right)^2}\right)\geq 0.
\end{align*}
Hence $f_2$ satisfies (\ref{Rc condi for f}) and (\ref{initial condition for f}).
\end{enumerate}

Now $f_1$ and $f_2$ are two different smooth functions satisfying (\ref{Rc condi for f}), and (\ref{initial condition for f}). Moreover, 
 \begin{align}
\varlimsup_{r\rightarrow\infty}\frac{\int_0^r f_i^2(s)ds}{r}= b, \quad i=1,2. \nonumber 
\end{align}

Therefore,  for noncompact Riemannian manifold $(M^3,g)$ with one end,  $\mathrm{V}_{M^3, 1}(M^3, g)= 2\pi b,\ (0<b\leq1)$ does not imply rigidity.
\begin{figure}[H]
    \centering
    \includegraphics{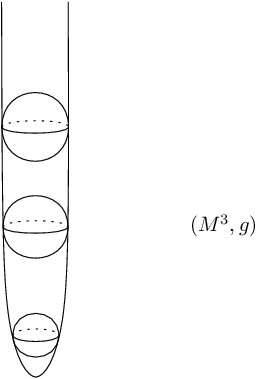}
    \caption{The model space with one end}
    \label{figure: The equality model in Cohn-Vossen's inequality-half-cylinder-3}
\end{figure}
}
\end{remark}

\section*{Acknowledgments}
We are indebted to Otis Chodosh for insightful comments.  We are grateful to Jianqiao Shang for carefully reading the earlier version of the paper and pointing out some gaps.  We particularly thank Jiaping Wang for his suggestion on the paper.  

\appendix

\section{The existence of smooth functions for $\mu$-bubbles}\label{sec appe exist}

\begin{lemma}\label{lem existence of bump function}
    For any closed subset $K\subsetneq (M^3,g)$, there is $f\in C^\infty(M^3)$ such that 
    \begin{align*}
        f|_{K}=0,\quad f|_{K^c}>0.
    \end{align*}
\end{lemma}

\begin{proof}
    Let $\{B_{p_i}(r_i)\}_{i=1}^\infty$ be a countable open cover of $K^c$ consisting of open balls, such that $\rho_i(x):=d(p_i,x)^2$ is smooth on each $B_{p_i}(r_i)$. Let 
    \begin{align*}
        f_i(x)=e^{-\frac{1}{r_i^2-\rho_i(x)}}1_{B_{p_i}(r_i)}
    \end{align*}
    be a smooth function which is positive on $B_{p_i}(r_i)$ and vanishes outside $B_{p_i}(r_i)$. For each $i$, define 
    \begin{align*}
        M_i=1+\sup\left\{\left|D^\alpha f_i(x)\right|:x\in M, \alpha=(\alpha_1,\alpha_2,\alpha_3)\in\mathbb{N}^3, |\alpha|=\alpha_1+\alpha_2+\alpha_3\leq i\right\}.
    \end{align*}
    We have $M_i<\infty$, since all the derivatives of $f_i$ vanishes outside a compact set.

    Let
    \begin{align*}
        f=\sum_{i=1}^\infty\frac{f_i}{2^iM_i}.
    \end{align*}
    Then for each multi-index $\alpha$, $|D^\alpha\frac{f_i}{2^iM_i}|\leq \frac{1}{2^i}$ when $i\geq|\alpha|$. So the series $\sum_{i=1}^\infty D^\alpha\frac{f_i}{2^iM_i}$ converges uniformly, which implies the smoothness of $f$. From the definition of $f$, we also have 
    \begin{align*}
        f|_{K}=0,\quad f|_{K^c}>0.
    \end{align*}
\end{proof}

\begin{lemma}\label{lem function for Caccioppoli set}
{For any $p\in (M^3, g)$ and $k\in \mathbb{Z}^+, L> 0$, there is $f\in C^\infty(M^3)$ with 
\begin{align}
f\big|_{\partial E_k}\leq kL, \quad f\big|_{\partial E_{k+ 1}}> (k+ \frac{1}{2})L, \quad  \text{and} \ (k+ \frac{1}{2})L\ \text{is a regular value of} \ f. \nonumber 
\end{align}
}
\end{lemma}

\begin{proof}
    From Lemma \ref{lem existence of bump function} there exists a smooth function $g$ on $M$ such that $g|_{\partial E_k}=0$ and $g|_{(\partial E_k)^c}>0$. Because $\partial E_{k+1}\subset \partial B_p((k+1)L)$ is compact, let $\displaystyle c:=\inf_{x\in\partial E_{k+1}} g(x)>0$. From Sard's theorem, $g$ has a regular value $c'\in (0,c)$. Then $f(x)=\frac{(k+1/2)L}{c'}g(x)$ is the desired smooth function.
\end{proof}

\begin{lemma}\label{lem existence of smooth function}
For any $\varepsilon>0$, there is a continuous function $\rho_k(x):\overline{N_k}\to [-L/2,L/2]$,  where $N_k= E_k- \overline{E_{k+ 1}}$, such that $\rho_k$ is smooth on $N_k$, $\rho_k(N_k)\subset(-L/2,L/2)$ and
\begin{align*}
    \rho_k|_{\partial E_k} = -L/2, \quad \rho_k|_{\partial E_{k+1}} = L/2, \quad |\nabla_M \rho_k|<(1+\varepsilon)^2.
\end{align*}
\end{lemma}

\begin{proof}
    Let $\delta>0$ such that $\frac{1}{1-2\delta}<1+\frac{\varepsilon}{2}$, and define $\tilde\rho_k$ on $\overline{N_k}\cup\Big(\overline{B_p((k+1+\delta)L)} - B_p((k-\delta)L)\Big)$ by
\begin{align*}
    \tilde\rho_k(x)=\left\{\begin{array}{ll}
     -\frac{L}{2}, & x\in \overline{B_p((k+\delta)L)}\\
     \frac{L}{2}, & x\in B_p((k+1-\delta)L)^c\\
     \frac{1}{1-2\delta}\big(d(p,x)-(k+\frac{1}{2}) L\big), & x\in \overline{B_p((k+1-\delta)L)}-B_p((k+\delta)L)
\end{array}\right..
\end{align*}
$\tilde\rho_k$ is a Lipschitz function with Lipschitz constant $\frac{1}{1-2\delta}$. Let $B_0^3(1)$ be the Euclidean ball in $\R^3$, $\eta(v)=Ce^{-\frac{1}{1-|v|^2}}1_{B_0^3(1)}$ be a smooth function on $\R^3$ and $\int_{\R^3}\eta(v)\d v=1$. Let $\alpha>0$ be sufficiently small such that $\exp_x$ is a diffeomorphism on $\{v\in T_x M:|v|<3\alpha\}$ for all $x\in \overline{B_p((k+1+\delta)L)}- B_p((k-\delta)L)$.
Then we define
\begin{align*}
    \rho_{k,\alpha}(x)=\frac{1}{\alpha^3}\int_{T_xM}\eta\Big(\frac{v}{\alpha}\Big)\tilde\rho_k(\exp_x (v))\d v.
\end{align*}
The integration is taken with respect to the Lebesgue measure on $T_xM=\R^3$

We claim that $\rho_{k,\alpha}(x)$ is smooth. Fix $x\in M$. For any smooth curve $\gamma:(-\alpha,\alpha)\to M$ with $\gamma(0)=x$, $\gamma$ can be written as $\gamma(t)=\exp_x(\tilde\gamma(t))$, where $\tilde\gamma$ is a smooth curve in $T_xM$. Let $\varphi_t(v)=\d\exp_x(\tilde\gamma(t))(v)$ be the linear isomorphism between $T_xM$ and $T_{\gamma(t)}M$. Then 
\begin{align*}
    \rho_{k,\alpha}(\gamma(t))=&\frac{1}{\alpha^3}\int_{T_{\gamma(t)}M}\eta\Big(\frac{v}{\alpha}\Big)\tilde\rho_k(\exp_{\gamma(t)} (v))\d v\\
    =&\frac{1}{\alpha^3}\int_{T_xM}\eta\Big(\frac{\varphi_t(v)}{\alpha}\Big)\tilde\rho_k(\exp_{\gamma(t)} (\varphi_t(v)))J_{\varphi_t}(v)\d v.
\end{align*}
Since $\exp_x(\tilde\gamma(t)+sv)$ is smooth about $(s,t)$ in a small neighborhood of $(0,0)$, $\varphi_t(v)=\d\exp_x(\tilde\gamma(t))(v)=\frac{\d}{\d s}|_{s=0}\exp_x(\tilde\gamma(t)+sv)$ is also smooth about $t$ in a small neighborhood of $0$. So $\eta\Big(\frac{\varphi_t(v)}{\alpha}\Big)$, $\exp_{\gamma(t)} (\varphi_t(v))$ and $J_{\varphi_t}(v)$ are smooth functions of $t$ near $0$. Then $\rho_{k,\alpha}(\gamma(t))$ is smooth about $t$ near $0$.  Hence $\rho_{k,\alpha}$ is smooth at $x$. 

We further assume $\exp_x(B_0^3(\alpha))\subset\overline{B_p((k+1+\delta)L)}- B_p((k+1-\delta)L)$ for all $x\in\partial B_p((k+1)L)$, and $\exp_x(B_0^3(\alpha))\subset\overline{B_p((k+\delta)L)}- B_p((k-\delta)L)$ for all $x\in\partial B_p(kL)$ by taking $\alpha$ small. When $x\in\partial B_p((k+1)L)$,
\begin{align*}
    \rho_{k,\alpha}(x)=\frac{1}{\alpha^3}\int_{T_xM}\eta\Big(\frac{v}{\alpha}\Big)\frac{L}{2}\d v=\frac{L}{2}.
\end{align*}
When $x\in\partial B_p(kL)$,
\begin{align*}
    \rho_{k,\alpha}(x)=\frac{1}{\alpha^3}\int_{T_xM}\eta\Big(\frac{v}{\alpha}\Big)\Big(-\frac{L}{2}\Big)\d v=-\frac{L}{2}.
\end{align*}
Moreover $-L/2\leq \rho_{k,\alpha}(x)\leq L/2$ on $N_k$.

Since $\rho_{k,\alpha}(x)$ converges to $\tilde\rho_k$ uniformly on $\overline{B_p((k+1+\delta)L)}- B_p((k-\delta)L)$, there is $\alpha_0$ such that the Lipschitz constants of  $\rho_{k,\alpha_0}$ and $\tilde\rho_k$ satisfy $\mathrm{Lip}(\rho_{k,\alpha_0})<\mathrm{Lip}(\tilde\rho_k)+\frac{\varepsilon}{2}<1+\varepsilon$.

Let $U=\{x\in N_k:\rho_{k,\alpha_0}(x)<0\}$, $V=\{x\in N_k:\rho_{k,\alpha_0}(x)>0\}$ be two open sets. From Lemma \ref{lem existence of bump function} there exists smooth bump functions $f$, $g$ on $M$ such that
\begin{align*}
    f|_{U}>0,\quad f|_{U^c}=0, \quad g|_{V}>0,\quad g|_{V^c}=0.
\end{align*}
And we may assume $f<L/4$, $g<L/4$, $|\nabla_M f|<\varepsilon/2$, $|\nabla_M g|<\varepsilon/2$ after a scaling.

Let
\begin{align*}
    \rho_k=\rho_{k,\alpha_0}|_{\overline{N_k}}+ f|_{\overline{N_k}}-g|_{\overline{N_k}}
\end{align*}
then $\rho_k$ is smooth, $\rho_k|_{\partial E_k} = -\frac{L}{2}$, $\rho_k|_{\partial E_{k+1}} = \frac{L}{2}$, $\rho_k(N_k)\subset(-L/2,L/2)$ and 
\begin{align*}
    |\nabla_M \rho_k|\leq\mathrm{Lip}(\rho_{k,\alpha_0})+|\nabla_M f|+|\nabla_M g|<(1+\varepsilon)^2.
\end{align*}
\end{proof}

\end{document}